\documentclass[11pt,reqno]{amsart}

\usepackage{fullpage, amsmath, amsthm, amstext, amssymb, amsfonts, fancyhdr, graphicx, tikz, forest, mathrsfs, array}
\usepackage[normalem]{ulem}
\usetikzlibrary{graphs,graphs.standard,quotes}

\numberwithin{equation}{section}

\newtheorem{theorem}{Theorem}[section]
\newtheorem{lemma}[theorem]{Lemma}

\newtheorem{corollary}[theorem]{Corollary}

\newtheorem{example}[theorem]{Example}
\newtheorem{definition}[theorem]{Definition}
\newtheorem{proposition}[theorem]{Proposition}

\theoremstyle{remark}

{\bf}{\rm}

\newcommand{\PF}{\text{PF}}
\newcommand{\Sym}{\mathfrak{S}}
\newcommand{\SH}{\text{S}}
\newcommand{\disp}{\text{disp}}
\newcommand{\cm}{\text{cm}}
\newcommand{\inv}{\text{inv}}
\newcommand{\PR}{\mathbb P}
\newcommand{\QR}{\mathbb Q}
\newcommand{\ER}{\mathbb E}
\newcommand{\FC}{\mathcal F}
\newcommand{\F}{\mathscr F}
\newcommand{\C}{\mathscr C}
\newcommand{\outcome}{\mathcal{O}}
\newcommand{\Var}{\mathbb V \mathrm{ar}}
\newcommand{\Cov}{\mathbb C \mathrm{ov}}
\newcommand{\be}{\begin{equation}}
\newcommand{\ee}{\end{equation}}

\def\weakto{\Longrightarrow}

\newcommand{\red}[1]{{\color{red} #1}}
\newcommand{\old}[1]{}

\NewDocumentCommand\DownArrow{O{2.0ex} O{black}}{%
   \mathrel{\tikz[baseline] \draw [<->, line width=0.5pt, #2] (0,0) -- ++(0,#1);}
}

\begin{document}

\forestset{parent color/.style args={#1}{
    {fill=#1},
    for tree={fill/. wrap pgfmath arg={#1!##1}{1/level()*80},draw=#1!80!black}},
    root color/.style args={#1}{fill={{#1!60!black!25},draw=#1!80!black}}
}

\title{Parking functions: From combinatorics to probability}

\author{Richard Kenyon}
\thanks{RK was supported by NSF DMS-1940932 and the Simons Foundation grant 327929}
\address{Department of Mathematics, Yale University, New Haven, CT 06520}
\email{richard.kenyon at yale.edu}

\author{Mei Yin}
\thanks{MY was supported by the University of Denver's Faculty Research Fund 84688-145601.}
\address{Department of Mathematics, University of Denver, Denver, CO 80208}
\email{mei.yin at du.edu}

\date{\today}

 % Large deviations

\begin{abstract}
Suppose that $m$ drivers each choose a preferred parking space in a linear car park with $n$ spots. In order, each driver goes to their chosen spot and parks there if possible, and otherwise takes the next available spot if it exists. If all drivers park successfully, the sequence of choices is called a parking function. Classical parking functions correspond to the case $m=n$; we study here combinatorial and probabilistic aspects of this generalized case.

We construct a family of bijections between parking functions $\PF(m, n)$ with $m$ cars and $n$ spots and spanning forests $\F(n+1, n+1-m)$ with $n+1$ vertices and $n+1-m$ distinct trees having specified roots. This leads to a bijective correspondence between $\PF(m, n)$ and monomial terms in the associated Tutte polynomial of a disjoint union of $n-m+1$ complete graphs. We present an identity between the ``inversion enumerator'' of spanning forests with fixed roots and the ``displacement enumerator'' of parking functions. The displacement is then related to the number of graphs on $n+1$ labeled vertices with a fixed number of edges, where the graph has $n+1-m$ disjoint rooted components with specified roots.

We investigate various probabilistic properties of a uniform parking function,
giving a formula for the law of a single coordinate. As a side result we obtain a recurrence relation for the displacement enumerator.  Adapting known results on random linear probes, we further deduce the covariance between two coordinates
when $m=n$.
\end{abstract}

\maketitle

\section{Introduction}\label{intro}
Parking functions were introduced by Konheim and Weiss \cite{KW} in the study of the linear probes of random hashing functions. Since then, parking functions have appeared all over combinatorics, probability, group theory, computer science, and beyond. The parking problem has counterparts in the enumerative theory of trees and forests (Chassaing and Marckert, \cite{CM}), in the analysis of set partitions and hyperplane arrangements (Stanley, \cite{Stanley1} \cite{Stanley2}), in the configuration of abelian sandpiles (Cori and Rossin, \cite{CR}), among others. We refer to Yan \cite{Yan} for a comprehensive survey.

Consider a parking lot with $n$ parking spots placed sequentially along a one-way street. A line of $m \leq n$ cars enters the lot, one by one. The $i$th car drives to its preferred spot $\pi_i$ and parks there if possible; if the spot is already occupied then the car parks in the first available spot after that. The list of preferences $\pi=(\pi_1,\dots,\pi_m)$ is called a \emph{generalized parking function} if all cars successfully park. (This generalizes the term
\emph{parking function} which classically refers to the case $m=n$. When there is no risk of confusion we will drop the modifier ``generalized" and simply refer to both of these cases as parking functions).  We denote the set of  parking functions by $\PF(m, n)$, where $m$ is the number of cars and $n$ is the number of parking spots. Using the pigeonhole principle, we see that a parking function $\pi \in \PF(m, n)$ must have at most one value $=n$, at most two values $\geq n-1$, and for each $k$ at most $k$ values $\ge n-k+1$, and any such function is a parking function. Equivalently, $\pi$ is a parking function if and only if
\begin{equation}\label{pigeon}
\#\{k: \pi_k \leq i\} \geq m-n+i, \hspace{.2cm} \forall i=n-m+1, \dots, n.
\end{equation}
Note that parking functions are invariant under the action of $\Sym_m$ by permuting cars.

The number of classical parking functions $|\PF(n, n)|$ is $(n+1)^{n-1}$ and coincides with the number of labeled trees on $n+1$ vertices. This combinatorial property motivated much work in the early study of parking functions. Many combinatorial bijections between the set of parking functions of length $n$ and labeled trees on $n+1$ vertices have been constructed.
They reveal deep connections between parking functions and other combinatorial structures. See Gilbey and Kalikow \cite{GK} for an extensive list of references.

Various generalizations of parking functions have been explored, for example, double parking functions (Cori and Poulalhon, \cite{CP1}), $k$-parking functions (Yan, \cite{Yan1}), and parking functions associated with an arbitrary vector $\vec{x}$ (Kung and Yan, \cite{KY1}): given a vector $\vec{x}=(x_1, \dots, x_m)$, an $\vec{x}$-parking function of length $m$ is a sequence $(a_1, \dots, a_m)$ whose non-decreasing rearrangement $(b_1, \dots, b_m)$ satisfies $b_i\leq x_1+\cdots+x_i$ for all $1\leq i\leq m$ (compare with Proposition \ref{inc} below). In \cite{PS1}, Pitman and Stanley related the number of $\vec{x}$-parking functions to the volume polynomials of certain types of polytopes. The generalized parking function investigated in this paper may be alternatively posed as an $\vec{x}$-parking function of length $m$, where the vector is $\vec{x}=(n-m+1, 1, \dots, 1)$.

Write $[n]$ for the set of integers $1, \dots, n$, and $[n]_0$ for the set of integers $0, 1, \dots, n$. While examining quotients of the polynomial ring, Postnikov and Shapiro \cite{PS} proposed the notion of $G$-parking functions associated with a general connected digraph $G$ with vertex set $[n]_0$. A $G$-parking function is a function $g$ from $[n] \rightarrow \mathbb{N}$, the set of non-negative integers, satisfying the following condition: For each subset $U \subseteq [n]$ of vertices of $G$, there exists a vertex $j \in U$ such that the number of edges from $j$ to vertices outside $U$ is greater than $g(j)$. We view vertex $0$ as the root with $g(0)=\infty$. The $G$-parking function generalizes the classical parking function since when $G=K_{n+1}$, the complete graph on $n+1$ vertices, the two definitions coincide. %Here we interpret $K_{n+1}$ as a digraph with directed edge $(i, j)$ for each $i\neq j$.
A family of bijections between the spanning trees of a graph $G$ and the set of $G$-parking functions was established by Chebikin and Pylyavskyy \cite{CP}.

Further along this direction, Kosti\'{c} and Yan \cite{KY} proposed the notion of $G$-multiparking functions. A $G$-multiparking function is a function $g$ from $[n]_0 \rightarrow \mathbb{N} \cup \{\infty\}$ such that for each subset $U \subseteq [n]_0$ of vertices of $G$, let $j$ be the vertex of smallest index in $U$;  either $g(j)=\infty$ or there exists a vertex $i \in U$ such that the number of edges from $i$ to vertices outside $U$ is greater than $g(i)$.
The vertices that satisfy $g(j)=\infty$ are viewed as roots. The $G$-multiparking functions with exactly one root (which is necessarily vertex $0$, as it is the vertex with smallest index in $U=[n]_0$) are exactly $G$-parking functions. Building upon the work of Gessel and Sagan \cite{GS1} on Tutte polynomials related to parking functions, Kosti\'{c} and Yan \cite{KY} described a family of bijections between the spanning forests of a graph $G$ and the set of $G$-multiparking functions.

\begin{lemma}\label{intro-lemma}
If $G=K_{n+1}$ with $n-m+1$ roots fixed at vertices $0, \dots, n-m$, then a $G$-multiparking function is a generalized parking function on $m$ cars and $n$ spots.
\end{lemma}

In other words, just as a classical parking function $\PF(n, n)$ may be viewed as a concrete realization of a $G$-parking function, a generalized parking function $\PF(m, n)$ may be viewed as a concrete realization of a modified $G$-multiparking function.

\begin{figure}
\begin{center}
\begin{tikzpicture}
  \graph[circular placement, radius=4cm,
         empty nodes, nodes={scale=0.3, circle, draw, fill=blue}] {
    \foreach \x in {a,...,f} {
      \foreach \y in {\x,...,f} {
        \x -- \y;
      };
    };
  };
\pgfmathparse{90}
\node at (\pgfmathresult:1.5cm) {$0/\infty$};
\pgfmathparse{150}
\node at (\pgfmathresult:1.6cm) {$1/\infty$};
\pgfmathparse{210}
\node at (\pgfmathresult:1.6cm) {$2/\infty$};
\pgfmathparse{270}
\node at (\pgfmathresult:1.5cm) {$3/3$};
\pgfmathparse{330}
\node at (\pgfmathresult:1.6cm) {$4/1$};
\pgfmathparse{30}
\node at (\pgfmathresult:1.6cm) {$5/3$};
\end{tikzpicture}
\end{center}
\caption{A $K_6$-multiparking function with $3$ roots. Each vertex is labeled by $i/g(i)$, i.e. the first label is the index of the vertex and the second label is the function value. The corresponding parking function is $(4, 2, 4) \in \PF(3, 5)$.}\label{multi}
\end{figure}
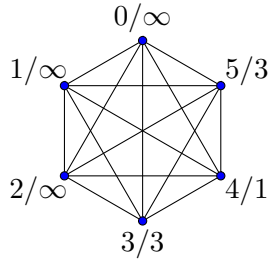

The correspondence between a $K_{n+1}$-multiparking function $g$ and a parking function $\pi \in \PF(m, n)$ is obtained via recording the values of $g$ at the non-root vertices of $K_{n+1}$ as a sequence $\pi=(g(n-m+1)+1, \dots, g(n)+1)$. See Figure \ref{multi}. We will present a proof of this correspondence in Section \ref{classical}.

This paper is organized as follows. Section \ref{comb} serves as a complement to existing combinatorial results on parking functions in the literature. Via the introduction of an auxiliary object, we give a family of bijections between parking functions $\PF(m, n)$ with $m$ cars and $n$ spots and spanning forests $\F(n+1, n+1-m)$ with $n+1$ vertices and $n+1-m$ distinct trees such that a specified set of $n+1-m$ vertices are the roots of the different trees (Theorems \ref{map} and \ref{construction}). This combinatorial construction is standard in nature, and parallel results may be found in Yan \cite{Yan1} for $\vec{x}$-parking functions and Kosti\'{c} and Yan \cite{KY} for $G$-multiparking functions. An important difference is that in our setting the roots are fixed. We extend the concept of critical left-to-right maxima and displacement for classical parking functions in Kosti\'{c} and Yan \cite{KY} to generalized parking functions, and establish a bijective correspondence between generalized parking functions $\PF(m, n)$ and monomial terms in the associated Tutte polynomial of the disjoint union of $n-m+1$ complete graphs whose combined number of vertices is $n+1$ (Theorem \ref{main}). This leads to an identity between the inversion enumerator of spanning forests with fixed roots and the displacement enumerator of generalized parking functions (Theorem \ref{inv-disp}). More interestingly, we relate the displacement of generalized parking functions to the number of graphs on $n+1$ labeled vertices with a fixed number of edges, where the graph has $n+1-m$ disjoint components with a specified vertex belonging to each component (Theorem \ref{disp}). This extends the corresponding connection between classical parking functions and connected graphs on $n+1$ labeled vertices with a fixed number of edges (Janson et al., \cite{JKLP}). Extending the method in Knuth \cite[Section 6.4]{Knuth2}, a bijective construction that carries the displacement of parking functions to the number of inversions of rooted forests is briefly described towards the end of this section.

In Section \ref{random}, we turn our attention back to the parking function itself and investigate various properties of a parking function chosen uniformly at random from $\PF(m, n)$. We illustrate the notion of parking function shuffle that decomposes a parking function into smaller components (Definition \ref{shuffle}). This construction leads to an explicit characterization of single coordinates $\pi_1\in[n]$. of random parking functions (Theorem \ref{main1} and Corollary \ref{component}). We examine the behavior of $\pi_1$ in the generic situation $m \lesssim n$, and find that on the left end it deviates from the constant value in a Poisson fashion while on the right end it approximates a Borel distribution with parameter $m/n$ (Corollaries \ref{boundary1} and \ref{boundary2}). This asymptotic tendency extends that in the special situation $m=n$ for classical parking functions, where the boundary behavior of single coordinates on the left and right ends both approach Borel$(1)$, as shown in Diaconis and Hicks \cite{DH}. As a side result of the shuffle construction, we also obtain a recurrence relation for the displacement enumerator of generalized parking functions (Proposition \ref{side}). We compute asymptotics
of all moments of single coordinates in Theorem \ref{mean}. Adapting known results on random linear probes (Theorem \ref{adapted}), we further deduce the covariance between two coordinates of random parking functions when $m=n$ (Proposition \ref{cov}). For $m \lesssim n$, the locations of unattempted parking spots have an impact on the random parking function, and we list a few interesting results (Proposition \ref{un} and Theorem \ref{Brownian}).

\section{Parking functions, spanning forests, and Tutte polynomials}\label{comb}
\subsection{Classical results}\label{classical}
The following proposition establishes the connection between generalized parking functions in our setting and $\vec{x}$-parking functions, where $\vec{x}=(n-m+1, 1, \dots, 1)$.
\begin{proposition}\label{inc}
Take a sequence $\pi=(\pi_1, \dots, \pi_m)$ with non-decreasing rearrangement $\lambda=(\lambda_1, \dots, \lambda_m)$. Then $\pi \in \PF(m, n)$ if and only if
\begin{equation}\label{mirror}
\lambda_i \le n-m+i, \hspace{.2cm} \forall i=1, \dots, m.
\end{equation}
\end{proposition}

\begin{proof}
This alternative criterion for parking functions involves a switch of coordinates, as illustrated in Figure \ref{pfineqs}.
\end{proof}

\begin{figure}[htbp]
\begin{center}\includegraphics[width=2.4in]{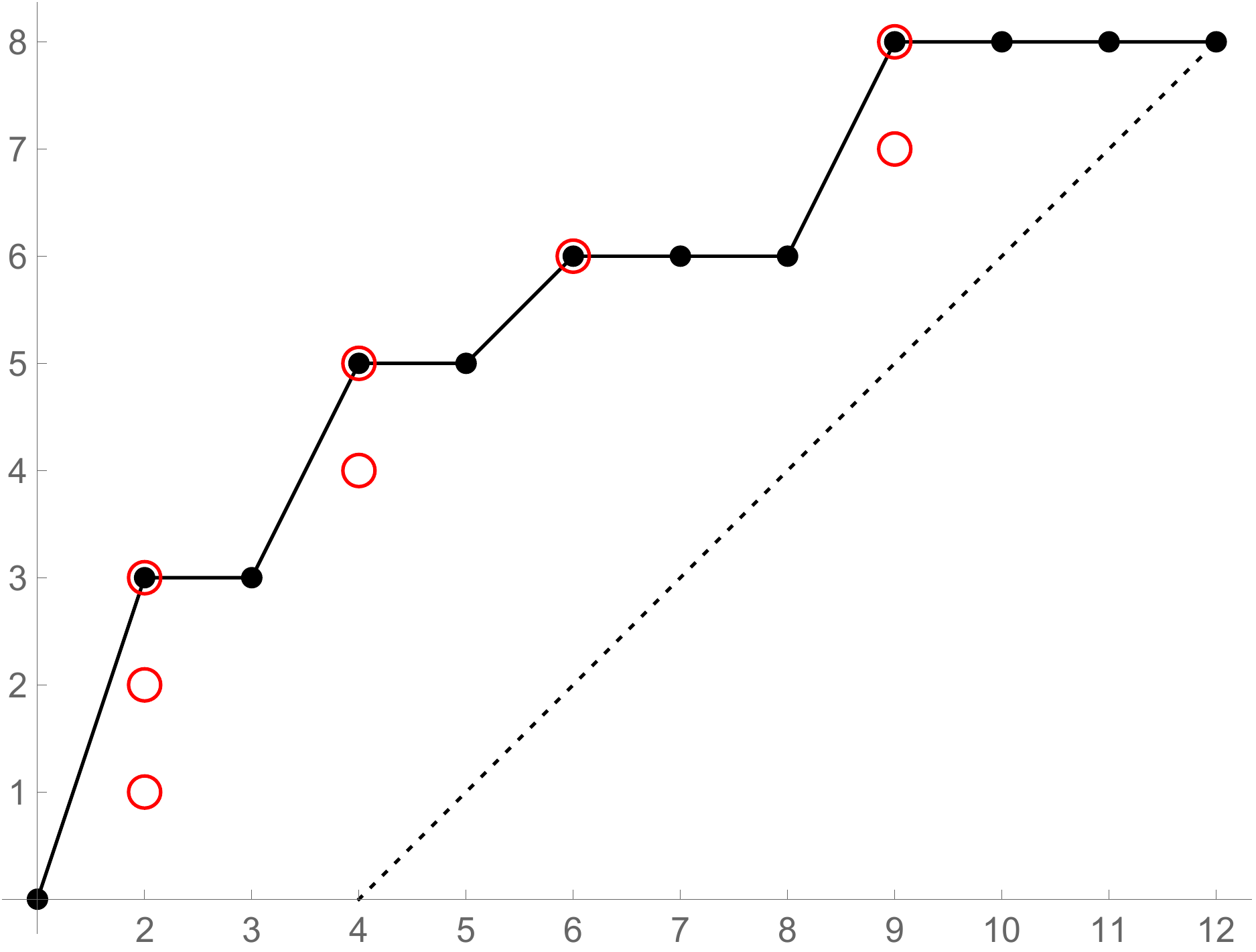}\end{center}
\caption{\label{pfineqs}
Equivalent inequalities for parking functions. For a parking function $\pi\in \PF(8,12)$ with non-decreasing rearrangement
$\lambda_1,\dots,\lambda_8 = 2,2,2,4,4,6,9,9$, we plotted points $(\lambda_j, j)$ which satisfy
(\ref{mirror}) in red and points $(i, \#\{k~|~\pi_k\le i\})$ which satisfy (\ref{pigeon}) in black. }
\end{figure}

We are now ready to present a proof for Lemma \ref{intro-lemma} stated in the Introduction.

\begin{proof}[Proof of Lemma \ref{intro-lemma}]
Consider an arbitrary subset $U \subseteq \{n-m+1, \dots, n\}$ of cardinality $i$. There exists a vertex $j \in U$ with $g(j)<n-i+1$, where $n-i+1$ is the number of edges from $j$ to vertices outside $U$, if and only if the non-decreasing rearrangement $\lambda$ of $g$ on $\{n-m+1, \dots, n\}$ satisfies $\lambda_{m-i+1}<n-i+1$, which is equivalent to $\lambda_{i}+1\leq n-m+i$. The conclusion then follows from Proposition \ref{inc}.
\end{proof}

\begin{theorem}[adapted from Pitman and Stanley \cite{PS1}]\label{number}
The number of parking functions $|\PF(m, n)|=(n-m+1)(n+1)^{m-1}$.
\end{theorem}

\begin{proof}
We extend Pollak's circle argument for classical parking functions \cite{Pollak} to this more general case. Add an additional space $n+1$, and arrange the spaces in a circle. Allow $n+1$ also as a preferred space. There are $(n+1)^m$ possible preference sequences $\pi$ for all $m$ cars, and $\pi$ is a parking function if and only if the spot $n+1$ is left open. For $k \in \mathbb{Z}/(n+1)\mathbb{Z}$, the preference sequence $\pi+k(1, \dots, 1)$ (modulo $n+1$) gives an assignment whose missing spaces are the rotations by $k$ of the missing spaces for the assignment of
$\pi$. Since there are $n-m+1$ missing spaces for the assignment of any preference sequence, any preference sequence $\pi$ has $n-m+1$ rotations which are parking functions. Therefore
\begin{equation}
|\PF(m, n)|=\frac{n-m+1}{n+1}(n+1)^m=(n-m+1)(n+1)^{m-1}.
\end{equation}
\end{proof}

The above classical results give us an effective method for generating a uniform random parking function: Choose an arbitrary preference sequence $\pi$ and then choose a random hole. Find $k$ so that $\pi+k(1, \dots, 1)$ sends this hole to $n+1$.

\subsection{One-to-one correspondence between parking functions and spanning forests}\label{pf}

The number $|\PF(m,n)|$ from Theorem \ref{number} in fact coincides with the number of rooted spanning forests $\F(n+1, n+1-m)$ with $n+1$ vertices and $n+1-m$ distinct trees such that a specified set of $n+1-m$ vertices are the roots of the different trees. Building upon the correspondence between classical parking functions and rooted spanning trees illustrated in Chassaing and Marckert \cite{CM} and Yan \cite{Yan}, we will construct
in this section an explicit bijection between generalized parking functions and such rooted spanning forests.

A parking function $\pi \in \PF(m, n)$ may be uniquely determined by its associated \emph{specification} $\vec{r}(\pi)$ and \emph{order permutation} $\sigma(\pi)$. Here the specification is $\vec{r}(\pi)=(r_1, \dots, r_n)$, where $r_k=\#\{i: \pi_i=k\}$ records the number of cars whose first preference is spot $k$. The order permutation $\sigma(\pi) \in \Sym_m$, on the other hand, is defined by
\begin{equation}
\sigma_i=|\{j: \pi_j<\pi_i, \text{ or } \pi_j=\pi_i \text{ and } j\leq i\}|,
\end{equation}
and so is the permutation that orders the list, without switching elements which are the same. For example, for $\pi=(4,2,2,4)$, $\sigma(\pi)=3124$. In words, $\sigma_i$ is the position of the entry $\pi_i$ in the non-decreasing rearrangement of $\pi$. We can easily recover a parking function $\pi$ by replacing $i$ in $\sigma(\pi)$ with the $i$th smallest term in the sequence $1^{r_1}\dots n^{r_n}$.

However, not every pair of a length $n$ vector $\vec{r}$ and a permutation $\sigma \in \Sym_m$ can be the specification and the order permutation of a parking function with $m$ cars and $n$ spots. The vector and the permutation must be compatible with each other, in the sense that the terms $1+\sum_{i=1}^{k-1} r_i, \dots, \sum_{i=1}^{k} r_i$ appear from left to right in $\sigma$ for every $k$ to satisfy the non-decreasing rearrangement requirement of $\pi$. Moreover, the specification $\vec{r}$ should satisfy a balance condition, equivalent to the sequence $1^{r_1}\dots n^{r_n}$ satisfying (\ref{mirror}). If we let
$k_i(\pi)$ for $i=1, \dots, n-m$ (and $k_0(\pi)=0$ and $k_{n-m+1}(\pi)=n+1$) represent the $n-m$ parking spots that are never attempted by any car, then
\begin{align}\label{determine1}
&\sum_{s=1}^{k_i} r_s=k_i-i, \hspace{.2cm} \forall 1\leq i\leq n-m, \\
&\sum_{s=1}^{j} r_s>j-i-1, \hspace{.2cm} \forall k_i<j<k_{i+1}, 0\leq i\leq n-m, \label{determine2} \\
&\sum_{s=1}^n r_s=m. \label{determine3}
\end{align}
Conversely, $\vec r$ is a specification of a parking function if there exist $n-m$ numbers $k_1, \dots, k_{n-m}$ with $0:=k_0<k_1<\cdots<k_{n-m}<k_{n-m+1}:=n+1$, satisfying the above.
These conditions are illustrated in a later example.

Let $\C(m, n)$ be the set of all compatible pairs.

\begin{theorem}\label{map}
The set $\C(m, n)$ is in one-to-one correspondence with $\PF(m, n)$.
\end{theorem}

\begin{proof}
This is immediate from the above construction.
\end{proof}

\begin{theorem}\label{construction}
The set $\C(m, n)$ is in one-to-one correspondence with $\F(n+1, n+1-m)$.
\end{theorem}

\begin{proof} We illustrate the proof with a representative
example. See Figure \ref{illustration1} representing an element of
$\F(13,4)$. We read the vertices in ``breadth first search" (BFS) order: $v_{01},\dots, v_{04}, v_5, \dots, v_{13} = 01,02,03,04,2,4,6,3,1,7,5,9,8$. That is, read root vertices in order first, then all vertices at level one (distance one
from a root), then those at level two (distance two from a root), and so on, where vertices at a given level are naturally ordered in order of increasing predecessor, and, if they have the same predecessor, increasing order.
We let $\sigma= 246317598$ be this vertex ordering once we remove the root vertices.
We let $r_i$ record the number of successors of $v_i$, that is, $\vec r=(2, 0, 1, 1, 0, 2, 0, 2, 0, 0, 1, 0)$.
Now $\sigma^{-1}=514273698$ is compatible with $\vec r$, by virtue of the fact that
vertices with the same predecessor are read in increasing order.

In order to show that $\vec r$ is balanced, we use a queue, where (starting from a queue containing the root vertices $v_{01},v_{02},v_{03},v_{04}$),
at each time step we
\begin{itemize}
\item read in the successors of the next $v_i$ (if the queue is empty) or
\item remove the top element of the queue then read in the successors of the next $v_i$.
\end{itemize}

See Figure \ref{illustration2}.
First we read in the successors of $v_{01}$, which are $2$ and $4$. Next, we remove the top element of the queue $2$ and read in the successors of $v_{02}$, which is none. Then, we remove the top element of the queue $4$ and read in the successors of $v_{03}$, which is $6$. Then, we remove the top element of the queue $6$ and read in the successors of $v_{04}$, which is $3$. Then, we remove the top element of the queue $3$ and read in the successors of $v_5=2$, which is none, so we end up with an empty queue. Then, we read in the successors of $v_6=4$, which are $1$ and $7$. And so on.

The queue length at time $k$ coincides with the number of cars that attempt to park at spot $k$ (whether successful or not), and the number of new vertices in the queue at time $k$ coincides with the number of cars $r_k$ whose first preference is spot $k$.
The time steps where the queue is empty correspond to the $k_i$, the unused spaces of the parking function.
In the example they are $5,10,12$.
These times satisfy (\ref{determine1}), and at times between $k_i$ and $k_{i+1}$ the $r_s$ satisfy
(\ref{determine2}). Note that (\ref{determine3}) follows by construction.
This completes the proof.

From Theorem \ref{map}, the corresponding generalized parking function is $\pi=(6, 1, 4, 1, 8, 3, 6, 11, 8)$. This parking function fills positions $[1,4],[6,9],[11]$.

\end{proof}

The above proof does not depend on using the BFS algorithm; any other algorithm which builds up a forest
one edge at a time through a sequence of growing subforests with the same roots will give an alternate bijection.
Generally, an algorithm checks the vertices of the forest one-by-one, starting with the ordered roots. At each step, we pick a new vertex and connect it to the checked vertices. The choice function (which defines the algorithm) tells us which new vertex to pick.

Another algorithm we will use below is ``BFS version II", which reads each rooted subtree in BFS order before moving on to
the next rooted subtree. In the example of Figure \ref{illustration1}, the BFS version II order is $01,2,4,1,7,02,03,6,04,3,5,9,8$.
Let us describe the queue procedure for version II in detail; see Figure \ref{illustration2}. The root vertices $v_{01}, v_{02}, v_{03}, v_{04}$ are implicitly in the queue (not recorded), but now they are not at the top of the queue, but interspersed. Explicitly, first the queue is empty, and we read in the successors of $v_{01}$, which are $2$ and $4$. Next, we remove the top element of the queue $2$ and read in the successors of $2$, which is none. Then, we remove the top element of the queue $4$ and read in the successors of $4$, which are $1$ and $7$. Then, we remove the top element of the queue $1$ and read in the successors of $1$, which is none. Then, we remove the top element of the queue $7$ and read in the successors of $7$, which is none, so we end up with an empty queue. Now that we have read the first subtree, we go on to the second subtree. We read in the successors of $v_{02}$, which is none, so we still have an empty queue. Then we move on to the third subtree. We read in the successors of $v_{03}$, which is $6$. And so on.

\begin{figure}
\begin{center}
\begin{forest}
[01, baseline, circle, fill={gray}[2, root color={red}][4, root color={red}[1, root color={red}][7, root color={red}]]]
\end{forest}
\quad
\begin{forest}
[02, baseline, circle, fill={gray}]
]
\end{forest}
\quad
\begin{forest}
[03, baseline, circle, fill={gray}[6, root color={red}]
]
\end{forest}
\quad
\begin{forest}[04, baseline, circle, fill={gray}[3, root color={red}[5, root color={red}[8, root color={red}]][9, root color={red}]]
]
\end{forest}
\caption{Rooted spanning forest.}\label{illustration1}
\end{center}
\end{figure}
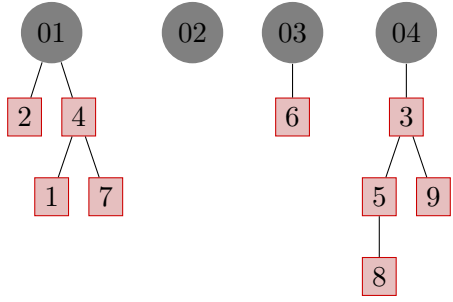

\begin{figure}
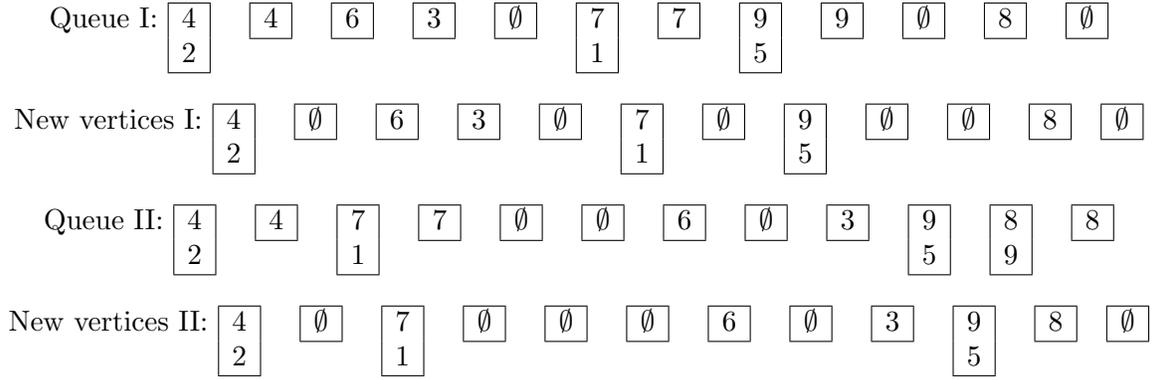

\begin{center}
Queue I:
\begin{tabular}[t]{|c|}\firsthline
4\\2\\\lasthline
\end{tabular}
\quad
\begin{tabular}[t]{|c|}\firsthline
4\\\lasthline
\end{tabular}
\quad
\begin{tabular}[t]{|c|}\firsthline
6\\\lasthline
\end{tabular}
\quad
\begin{tabular}[t]{|c|}\firsthline
3\\\lasthline
\end{tabular}
\quad
\begin{tabular}[t]{|c|}\firsthline
$\emptyset$\\\lasthline
\end{tabular}
\quad
\begin{tabular}[t]{|c|}\firsthline
7\\1\\\lasthline
\end{tabular}
\quad
\begin{tabular}[t]{|c|}\firsthline
7\\\lasthline
\end{tabular}
\quad
\begin{tabular}[t]{|c|}\firsthline
9\\5\\\lasthline
\end{tabular}
\quad
\begin{tabular}[t]{|c|}\firsthline
9\\\lasthline
\end{tabular}
\quad
\begin{tabular}[t]{|c|}\firsthline
$\emptyset$\\\lasthline
\end{tabular}
\quad
\begin{tabular}[t]{|c|}\firsthline
8\\\lasthline
\end{tabular}
\quad
\begin{tabular}[t]{|c|}\firsthline
$\emptyset$\\\lasthline
\end{tabular}

\vskip.1truein

New vertices I:
\begin{tabular}[t]{|c|}\firsthline
4\\2\\\lasthline
\end{tabular}
\quad
\begin{tabular}[t]{|c|}\firsthline
$\emptyset$\\\lasthline
\end{tabular}
\quad
\begin{tabular}[t]{|c|}\firsthline
6\\\lasthline
\end{tabular}
\quad
\begin{tabular}[t]{|c|}\firsthline
3\\\lasthline
\end{tabular}
\quad
\begin{tabular}[t]{|c|}\firsthline
$\emptyset$\\\lasthline
\end{tabular}
\quad
\begin{tabular}[t]{|c|}\firsthline
7\\1\\\lasthline
\end{tabular}
\quad
\begin{tabular}[t]{|c|}\firsthline
$\emptyset$\\\lasthline
\end{tabular}
\quad
\begin{tabular}[t]{|c|}\firsthline
9\\5\\\lasthline
\end{tabular}
\quad
\begin{tabular}[t]{|c|}\firsthline
$\emptyset$\\\lasthline
\end{tabular}
\quad
\begin{tabular}[t]{|c|}\firsthline
$\emptyset$\\\lasthline
\end{tabular}
\quad
\begin{tabular}[t]{|c|}\firsthline
8\\\lasthline
\end{tabular}\quad
\begin{tabular}[t]{|c|}\firsthline
$\emptyset$\\\lasthline
\end{tabular}

\vskip.1truein

Queue II:
\begin{tabular}[t]{|c|}\firsthline
4\\2\\\lasthline
\end{tabular}
\quad
\begin{tabular}[t]{|c|}\firsthline
4\\\lasthline
\end{tabular}
\quad
\begin{tabular}[t]{|c|}\firsthline
7\\1\\\lasthline
\end{tabular}
\quad
\begin{tabular}[t]{|c|}\firsthline
7\\\lasthline
\end{tabular}
\quad
\begin{tabular}[t]{|c|}\firsthline
$\emptyset$\\\lasthline
\end{tabular}
\quad
\begin{tabular}[t]{|c|}\firsthline
$\emptyset$\\\lasthline
\end{tabular}
\quad
\begin{tabular}[t]{|c|}\firsthline
6\\\lasthline
\end{tabular}
\quad
\begin{tabular}[t]{|c|}\firsthline
$\emptyset$\\\lasthline
\end{tabular}
\quad
\begin{tabular}[t]{|c|}\firsthline
3\\\lasthline
\end{tabular}
\quad
\begin{tabular}[t]{|c|}\firsthline
9\\5\\\lasthline
\end{tabular}
\quad
\begin{tabular}[t]{|c|}\firsthline
8\\9\\\lasthline
\end{tabular}
\quad
\begin{tabular}[t]{|c|}\firsthline
8\\\lasthline
\end{tabular}

\vskip.1truein

New vertices II:
\begin{tabular}[t]{|c|}\firsthline
4\\2\\\lasthline
\end{tabular}
\quad
\begin{tabular}[t]{|c|}\firsthline
$\emptyset$\\\lasthline
\end{tabular}
\quad
\begin{tabular}[t]{|c|}\firsthline
7\\1\\\lasthline
\end{tabular}
\quad
\begin{tabular}[t]{|c|}\firsthline
$\emptyset$\\\lasthline
\end{tabular}
\quad
\begin{tabular}[t]{|c|}\firsthline
$\emptyset$\\\lasthline
\end{tabular}
\quad
\begin{tabular}[t]{|c|}\firsthline
$\emptyset$\\\lasthline
\end{tabular}
\quad
\begin{tabular}[t]{|c|}\firsthline
6\\\lasthline
\end{tabular}
\quad
\begin{tabular}[t]{|c|}\firsthline
$\emptyset$\\\lasthline
\end{tabular}
\quad
\begin{tabular}[t]{|c|}\firsthline
3\\\lasthline
\end{tabular}
\quad
\begin{tabular}[t]{|c|}\firsthline
9\\5\\\lasthline
\end{tabular}
\quad
\begin{tabular}[t]{|c|}\firsthline
8\\\lasthline
\end{tabular}\quad
\begin{tabular}[t]{|c|}\firsthline
$\emptyset$\\\lasthline
\end{tabular}
\caption{Successive states of the queue (left to right) for the BFS algorithm, version I and version II.
For version II of the BFS algorithm, the queue is empty at times $5, 6, 8$, separating the $12$ available spots into disjoint segments of $[1, 4], [7], [9, 12]$, with respective length $4, 1, 4$. This coincides with the number of non-root vertices in the left-to-right trees of the rooted forest.}\label{illustration2}
\end{center}
\end{figure}

For BFS version I, as mentioned in the proof,
the queue length at time $k$ coincides with the number of cars $y_k(\pi)$ that attempt to park at spot $k$ (whether successful or not), and the number of new vertices in the queue at time $k$ coincides with the number of cars $r_k(\pi)$ whose first preference is spot $k$.
We have
\begin{equation}
y_k(\pi)=\left\{
           \begin{array}{ll}
             r_k(\pi) & \hbox{if $y_{k-1}(\pi)=0$,} \\
             y_{k-1}(\pi)-1+r_k(\pi) & \hbox{otherwise.} \\
           \end{array}
         \right.
\end{equation}
In particular, the number of times when the queue is empty corresponds with the number of spots that are never attempted by any car, and is also the number of tree components of the spanning forest minus one. This explains the balance condition at the beginning of this section. Upon further derivation, we have
\begin{equation}
y_j(\pi)=\#\{s: k_i <\pi_s \leq j\}-(j-k_i-1), \hspace{.2cm} \forall k_i<j \leq k_{i+1}, 0\leq i\leq n-m.
\end{equation}

\begin{proposition}\label{un2}
The number of parking functions $\pi \in \PF(m, n)$ subject to the constraint that fixed $n-m$ spots $k_1, \dots, k_{n-m}$ with $0:=k_0<k_1<\cdots<k_{n-m}<k_{n-m+1}:=n+1$ are unattempted by any car is given by
\begin{equation}
\prod_{i=0}^{n-m} (k_{i+1}-k_i)^{k_{i+1}-k_i-2} \binom{m}{k_1-k_0-1, \dots, k_{n-m+1}-k_{n-m}-1}.
\end{equation}
\end{proposition}

\begin{proof}
Recall that the number of classical parking functions of length $n$ is $(n+1)^{n-1}$. Since there are $n-m$ parking spots that are never attempted by any car, the parking function $\pi \in \PF(m, n)$ is separated into $n-m+1$ disjoint non-interacting segments (some segments might be empty), with each segment a classical parking function of length $(k_{i+1}-k_i-1)$ after translation. The multinomial coefficient then comes from considering different ways of permuting the segments.
\end{proof}

\subsection{One-to-one correspondence between parking functions and monomial terms in the Tutte polynomial}\label{Tutte}
Given a classical parking function $\pi \in \PF(n, n)$, there are two statistics that capture important features of $\pi$. The first one is the number of \textit{critical left-to-right maxima $\cm(\pi)$}. We say that a term $\pi_i=j$ is a left-to-right maximum if $\pi_s<j$ for all $s<i$, and we say that a term $\pi_i=j$ is critical if there are exactly $j-1$ terms less than $j$ and exactly $n-j$ terms greater than $j$ in $\pi$, so in particular, the entry $j$ is unique in $\pi$. For example, for $\pi=(4, 1, 1, 5, 3) \in \PF(5, 5)$, the subscripts $1$ and $4$ are critical left-to-right maxima, and $\cm(\pi)=2$. The second is the \textit{total displacement $\disp(\pi)$}, which corresponds to the total number of extra spaces cars are required to travel past their first preferred spot. For example, for $\pi=(1, 3, 5, 1, 3) \in \PF(5, 5)$, the parking outcome is $\outcome(\pi)=(1, 3, 5, 2, 4)$, and the displacement is $\disp(\pi)=2$. Depending on the context, displacement may be referred to differently as area, inconvenience, reversed sum, and so forth. It is associated with the number of linear probes in hashing functions (Knuth, \cite{Knuth1}), the number of inversions in labeled trees on $n+1$ vertices (Kreweras, \cite{Kr}), and the number of hyperplanes separating a given region from the base region in the extended Shi arrangement (Stanley, \cite{Stanley2}), to name a few.

Building on earlier work of Spencer \cite{Sp} and Chebikin and Pylyavskyy \cite{CP}, Kosti\'{c} and Yan \cite{KY} established an explicit expression of $T_{K_{n+1}}(x, y)$, the Tutte polynomial of complete graphs with $n+1$ vertices, in terms of the above two statistics of classical parking functions:
\begin{equation}\label{mono}
T_{K_{n+1}}(x, y)=\sum_{\pi \in \PF(n, n)} x^{\cm(\pi)}y^{\disp(\pi)}.
\end{equation}
There is also follow-up work from Chang et al. \cite{CMY}, who demonstrated the bijection between classical parking functions and monomial terms in the Tutte polynomial through an equivalent notion to critical maxima termed critical-bridge, avoiding any use of spanning trees. Their results imply that the number of internally active edges and the number of externally active edges of a spanning tree is equidistributed with the number of critical left-to-right maxima and the total displacement of the associated classical parking function.

In this section, we will extend the concept of critical left-to-right maxima and displacement for classical parking functions to generalized parking functions, and build a correspondence between generalized parking functions and monomial terms in the associated Tutte polynomial of the disjoint union of complete graphs. Displacement for a generalized parking function is defined analogously as for a classical parking function. For example, for $\pi=(1, 3, 5, 1) \in \PF(4, 5)$, the parking outcome $\outcome(\pi)=(1, 3, 5, 2)$, giving $\disp(\pi)=1$. The definition of
critical left-to-right maxima is more complicated. For a parking function $\pi \in \PF(m, n)$, as was observed in Section \ref{pf}, there are $n-m$ parking spots that are never attempted by any car. This separates $\pi$ into disjoint segments, and the parking preferences and hence outcomes are independent across different segments once the unattempted spots are fixed. We may view each segment of $\pi$ as a classical parking function by itself after translation, and consider the corresponding left-to-right maxima associated with the segment. Alternatively, we notice that the concept of critical left-to-right maxima is essentially built on the relative order between the coordinates, and not the coordinate value itself. For example, for $\pi=(4, 3, 5, 5) \in \PF(4, 6)$, the relative ordering of the coordinates is $\pi'=(2, 1,3, 3)$, which may equivalently be obtained through translation by $2$ from $\pi$, and the subscript $1$ is the only critical left-to-right maximum (thinking of $\pi'\in\PF(4,4)$), with $\cm(\pi)=1$. We define $\cm(\pi)$ of a generalized parking function to be the sum of the number of critical left-to-right maxima across all the segments. In like manner, $\disp(\pi)$ may also be considered as a sum of the total displacement across all the segments. Using the BFS algorithm version II, each translated segment of $\pi$ corresponds to a tree component in the rooted spanning forest (see Theorems \ref{map} and \ref{construction}), and also a monomial term in the Tutte polynomial (using (\ref{mono})). Since the trees are disjoint, $\pi$ then corresponds to the product of the respective Tutte mononomials.

We illustrate this bijection in more detail with the representative example from Section \ref{pf}.
We read the rooted spanning forest in Figure \ref{illustration1} using version II of the BFS algorithm:
$$v_{01}, v_1, \dots, v_4, v_{02}, v_{03}, v_5, v_{04}, v_6, \dots, v_9= 01,2,4,1,7,02,03,6,04,3,5,9,8.$$ We let $\sigma= 241763598$ be this vertex ordering once we remove the root vertices, and let $r_i$ record the number of successors of $v_i$, that is, $\vec r=(2,0,2,0,0,0,1,0,1,2,1,0)$. Now $\sigma^{-1}=316275498$ is compatible with $\vec r$, and the corresponding generalized parking function is $\pi=(3, 1, 9, 1, 10, 7, 3, 11, 10)$.

There are several equivalent ways for us to retrieve the different segments of $\pi$, corresponding to the non-root vertices within different tree components. Maybe the fastest approach is by realizing that spots $5, 6, 8$ are never attempted, thus giving $\pi_1=(3, 1, 1, 3)$ for tree $01$, tree $02$ consists of an isolated root only, $\pi_2=(7)$ for tree $03$, and $\pi_3=(9, 10, 11, 10)$ for tree $04$.
% Alternatively, the inverse of the order permutation is $\sigma=241763598$, we may read the respective subscripts of $\pi$ in this order and extract values $1, 1, 3, 3, 7, 9, 10, 10, 11$ in sequence.
Then either translating the different segments of $\pi$ to classical parking functions, respectively $\pi'_1=(3, 1, 1, 3)$, $\pi'_2=(1)$, $\pi'_3=(1, 2, 3, 2)$, or using the generalized definition of critical left-to-right maxima and displacement introduced in the previous paragraph, we obtain the corresponding Tutte monomial $x^2y^4$ for the forest. We remark that if we replace the non-root vertices of the spanning forest with their respective ordering within the tree, then the translated segments $\pi'_i$ of $\pi_i$ may also be recovered directly using the correspondence established in Theorems \ref{map} and \ref{construction}. See Figure \ref{illustration3}.

\begin{figure}
\begin{center}
\colorbox{cyan}{$\pi=(3, 1, 9, 1, 10, 7, 3, 11, 10)\longleftrightarrow x^2y^4$}\\
\vskip.1truein
\begin{forest}
[01, baseline, circle, fill={gray}[2, root color={red}][3, root color={red}[1, root color={red}][4, root color={red}]]]
\end{forest}
\quad
\begin{forest}
[02, baseline, circle, fill={gray}]
]
\end{forest}
\quad
\begin{forest}
[03, baseline, circle, fill={gray}[1, root color={red}]
]
\end{forest}
\quad
\begin{forest}[04, baseline, circle, fill={gray}[1, root color={red}[2, root color={red}[3, root color={red}]][4, root color={red}]]
]
\end{forest}

\begin{tabular}[t]{cccccccccccccc}
$\DownArrow[30pt][>=latex,orange, ultra thick]$ & $\DownArrow[30pt][>=latex,orange, ultra thick]$ & $\DownArrow[30pt][>=latex,orange, ultra thick]$ \\
$\pi_1=(3, 1, 1, 3)$ & $\pi_2=(7)$ & $\pi_3=(9, 10, 11, 10)$\\
$\pi'_1=(3, 1, 1, 3)$ & $\pi'_2=(1)$ & $\pi'_3=(1, 2, 3, 2)$\\
$\DownArrow[30pt][>=latex,orange, ultra thick]$ & $\DownArrow[30pt][>=latex,orange, ultra thick]$ & $\DownArrow[30pt][>=latex,orange, ultra thick]$ \\
$y^2$ & $x$ & $xy^2$
\end{tabular}
\caption{Bijection between parking function and monomial term of the Tutte polynomial. Non-root vertices of the spanning forest are replaced with their respective ordering within the tree.}\label{illustration3}
\end{center}
\end{figure}
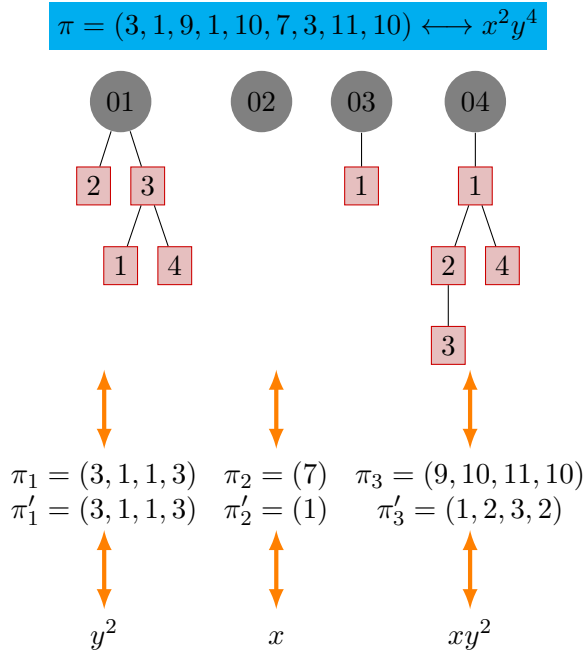

For notational convenience, we denote by $s=n-m+1$ in the following theorems.

\begin{theorem}\label{main}
For any positive integer $s$ and any nonnegative integer $m$,
\begin{equation}
\sum_{x_1+\cdots+x_s=m} \binom{m}{x_1, \dots, x_s} T_{K_{x_1+1}}(x, y)\cdots T_{K_{x_s+1}}(x, y)=\sum_{\pi \in \PF(m, m+s-1)} x^{\cm(\pi)}y^{\disp(\pi)},
\end{equation}
where $x_i\geq 0$, $\forall 1\leq i\leq s$, and $\cm(\cdot)$ and $\disp(\cdot)$ for a parking function $\pi$ are defined as above.
\end{theorem}

\begin{proof}
The multinomial sum describes the different ways the tree components of the forest are built. Recall that the number of critical left-to-right maxima and total displacement both factor over the trees, and are non-interacting when the tree sizes are fixed. Extending (\ref{mono}), the statement then follows from the construction explained above.
\end{proof}

For a rooted spanning forest, an inversion is a pair $(i, j)$ for which $i<j$ are in the same tree component and $j$ lies on the unique path connecting the root to $i$. The next theorem connects the inversion enumerator of spanning forests with the displacement enumerator of generalized parking functions. Related work on spanning trees vs. classical parking functions may be found in Kreweras \cite{Kr} and Gessel and Wang \cite{GW}.

\begin{theorem}\label{inv-disp}
Take any positive integer $s$ and any nonnegative integer $m$. Define the displacement enumerator of parking functions with $m$ cars and $m+s-1$ spots as
\begin{equation}\label{displacement}
D_{m, s}(y)=\sum_{\pi \in \PF(m, m+s-1)} y^{\disp(\pi)},
\end{equation}
and the inversion enumerator of spanning forests with $m+s$ vertices and $s$ fixed roots (the
roots are taken to be the $s$ smallest vertices) as
\begin{equation}
I_{m ,s}(y)=\sum_{F \in \F(m+s, s)} y^{\inv(F)}.
\end{equation}
Then $D_{m, s}(y)=I_{m, s}(y)$.
\end{theorem}

\begin{proof}
Setting $x=1$ in Theorem \ref{main}, we have
\begin{multline}
\sum_{\pi \in \PF(m, m+s-1)} y^{\disp(\pi)}=\sum_{x_1+\cdots+x_s=m} \binom{m}{x_1, \dots, x_s} T_{K_{x_1+1}}(1, y)\cdots T_{K_{x_s+1}}(1, y)\\
=\sum_{x_1+\cdots+x_s=m} \binom{m}{x_1, \dots, x_s} \sum_{F_1 \in \F(x_1+1, 1)}y^{\inv(F_1)}\cdots \sum_{F_s \in \F(x_s+1, 1)}y^{\inv(F_s)}
=\sum_{F \in \F(m+s, s)} y^{\inv(F)},
\end{multline}
where the last equality takes into consideration the different ways of constructing a forest with fixed roots and a specified number of non-root vertices and factors the inversion number of the forest across its tree components, and the second-to-last equality uses the correspondence between Tutte polynomial on a complete graph with $n+1$ vertices and the inversion enumerator of spanning trees with $n+1$ vertices (i.e. a one-rooted forest),
\begin{equation}\label{MR}
\sum_{F \in \F(n+1, 1)}y^{\inv(F)}=T_{K_{n+1}}(1, y)=\sum_{\pi \in \PF(n, n)} y^{\disp(\pi)},
\end{equation}
established in Mallows and Riordan \cite{MR}.
\end{proof}

\begin{corollary}
For any positive integer $s$ and any nonnegative integer $m$,
\begin{equation}
\sum_{x_1+\cdots+x_s=m} \binom{m}{x_1, \dots, x_s} (x_1+1)^{x_1-1}\cdots(x_s+1)^{x_s-1}=s(m+s)^{m-1},
\end{equation}
where $x_i\geq 0$, $\forall 1\leq i\leq s$.
\end{corollary}

\begin{proof}
This follows immediately when we set $x=y=1$ in Theorem \ref{main}. We recognize that both sides of the equation represent the number of spanning forests with $m+s$ vertices and $s$ distinct trees with specified roots.
\end{proof}

Taking $k$th order derivatives of (\ref{MR}) and setting $y=1$, we have
\begin{equation}
\sum_{\pi \in \PF(n, n)}\binom{\disp(\pi)}{k}=\frac{1}{k!}\frac{d^k}{dy^k}\left. T_{K_{n+1}}(1, y)\right\vert_{y=1}.
\end{equation}
In Janson et al. \cite{JKLP}, it was shown that the sum of $\binom{\disp(\pi)}{k}$ taken over all classical parking functions of length $n$ is equal to the total number of connected graphs with $n+k$ edges on $n+1$ labeled vertices. Building upon Theorem \ref{main}, we will extend this connection between classical parking functions and connected graphs to one that involves generalized parking functions and disjoint union of connected graphs.

\begin{theorem}\label{disp}
The sum of $\binom{\disp(\pi)}{k}$ taken over all parking functions $\pi \in \PF(m, m+s-1)$ is equal to the total number of graphs with $m+k$ edges on $m+s$ labeled vertices such that the graph has $s$ disjoint components with a specified vertex belonging to each component.
\end{theorem}

\begin{proof}
Taking $k$th order derivatives in Theorem \ref{main} and setting $y=1$, we have
\begin{equation}\label{last}
\sum_{\pi \in \PF(m, m+s-1)} \binom{\disp(\pi)}{k}=\sum_{x_1+\cdots+x_s=m} \binom{m}{x_1, \dots, x_s} \sum_{t_1+\cdots+t_s=k}\prod_{i=1}^s \frac{1}{(t_i)!}\frac{d^{t_i}}{dy^{t_i}}\left.T_{K_{x_i+1}}(1, y)\right\vert_{y=1}.
\end{equation}
As was explained above,
\begin{equation}
\frac{1}{(t_i)!}\frac{d^{t_i}}{dy^{t_i}}\left.T_{K_{x_i+1}}(1, y)\right\vert_{y=1}
\end{equation}
gives the total number of connected graphs with $x_i+t_i$ edges on $x_i+1$ labeled vertices. The conclusion follows when we note that the multinomial sum on the right of (\ref{last}) takes into account the different ways of adding $m$ labeled vertices and $k$ additional edges to the disjoint components that each contains one of the $s$ specified vertices at the start.
\end{proof}

\subsection{Another bijective construction}
As nice as it is, there is one disadvantage of the bijective correspondence established using the BFS algorithm in Section \ref{pf} between generalized parking functions and rooted spanning forests: The total displacement of a parking function $\pi$ is not necessarily equal to the inversion number of its associated spanning forest $F$. This shortcoming may be remedied by extending a bijective construction in Knuth \cite[Section 6.4]{Knuth2} between classical parking functions and rooted spanning trees. Building upon Knuth's auxiliary tree procedure, we briefly explain our auxiliary forest procedure below, again with the representative example from Section \ref{pf}.

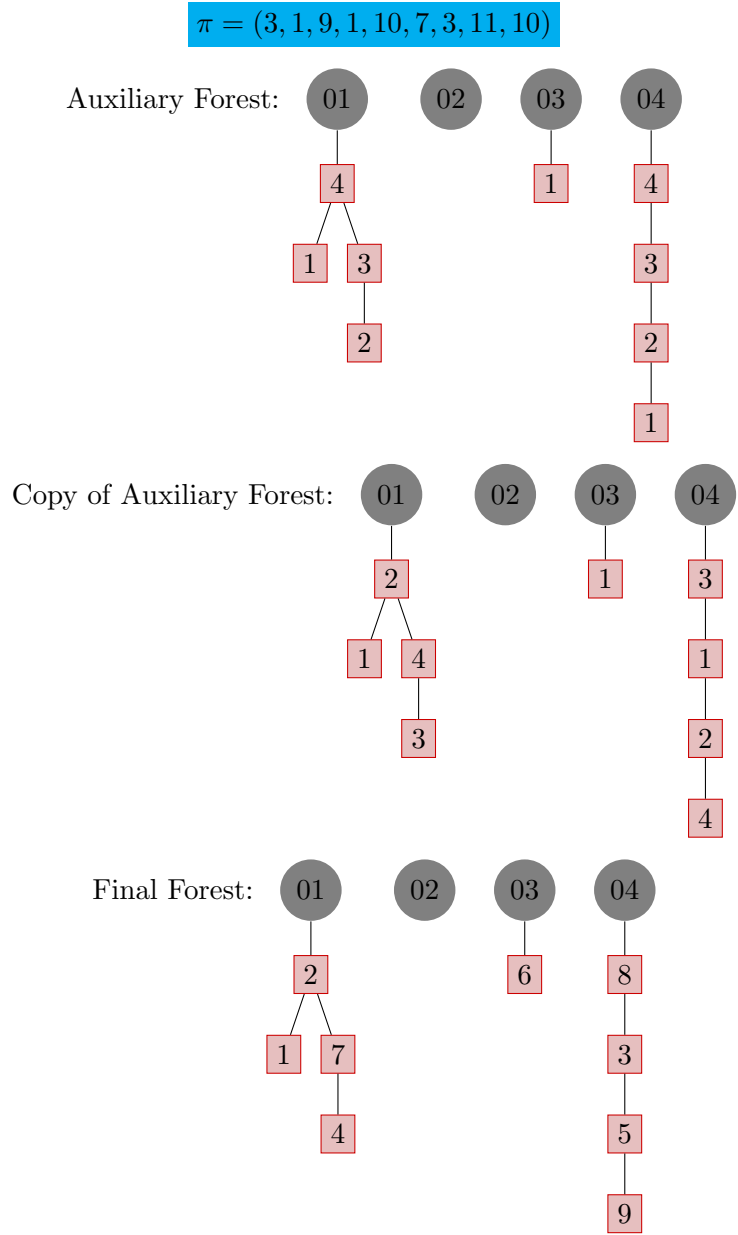
\begin{figure}
\begin{center}
\colorbox{cyan}{$\pi=(3, 1, 9, 1, 10, 7, 3, 11, 10)$}\\

\vskip.1truein

Auxiliary Forest:
\begin{forest}
[01, baseline, circle, fill={gray}[4, root color={red}[1, root color={red}][3, root color={red}[2, root color={red}]]]]
\end{forest}
\quad
\begin{forest}
[02, baseline, circle, fill={gray}]
]
\end{forest}
\quad
\begin{forest}
[03, baseline, circle, fill={gray}[1, root color={red}]
]
\end{forest}
\quad
\begin{forest}[04, baseline, circle, fill={gray}[4, root color={red}[3, root color={red}[2, root color={red}[1, root color={red}]]]]]
]
\end{forest}

\vskip.1truein

Copy of Auxiliary Forest:
\begin{forest}
[01, baseline, circle, fill={gray}[2, root color={red}[1, root color={red}][4, root color={red}[3, root color={red}]]]]
\end{forest}
\quad
\begin{forest}
[02, baseline, circle, fill={gray}]
]
\end{forest}
\quad
\begin{forest}
[03, baseline, circle, fill={gray}[1, root color={red}]
]
\end{forest}
\quad
\begin{forest}[04, baseline, circle, fill={gray}[3, root color={red}[1, root color={red}[2, root color={red}[4, root color={red}]]]]]
]
\end{forest}

\vskip.1truein

Final Forest:
\begin{forest}
[01, baseline, circle, fill={gray}[2, root color={red}[1, root color={red}][7, root color={red}[4, root color={red}]]]]
\end{forest}
\quad
\begin{forest}
[02, baseline, circle, fill={gray}]
]
\end{forest}
\quad
\begin{forest}
[03, baseline, circle, fill={gray}[6, root color={red}]
]
\end{forest}
\quad
\begin{forest}[04, baseline, circle, fill={gray}[8, root color={red}[3, root color={red}[5, root color={red}[9, root color={red}]]]]]
]
\end{forest}
\caption{Rooted spanning forest constructed via an auxiliary forest. The total displacement of the parking function coincides with the inversion number of the forest.}\label{illustration4}
\end{center}
\end{figure}

Let $\pi \in \PF(m ,n)$ be a parking function with parking outcome $\outcome(\pi)$. Recall that $\pi$ and $\outcome(\pi)$ may be decomposed into disjoint segments, with each segment by itself a classical parking function (respectively, parking outcome) after translation. We denote the translated segment of the parking function corresponding to tree $j$ by $\pi'_j$ and the outcome by $\outcome(\pi'_j)$, and further denote the inverse permutation of $\outcome(\pi'_j)$ by $\tau_j=((\tau_{j})_1, \dots, (\tau_{j})_{l_j})$, where $l_j$ is the length of $\tau_j$. For our example, using BFS version II, $\pi'_1=(3, 1, 1, 3)$, $\pi'_2=(1)$, $\pi'_3=(1, 2, 3, 2)$, with respective parking outcome $\outcome(\pi'_1)=(3, 1, 2, 4)$, $\outcome(\pi'_2)=(1)$, $\outcome(\pi'_3)=(1, 2, 3, 4)$, and inverse $\tau_1=(2, 3, 1, 4)$, $\tau_2=(1)$, $\tau_3=(1, 2, 3, 4)$. For each tree $j$, we construct an auxiliary tree by letting the predecessor of vertex $k$ be the first element on the right of $k$ and larger than $k$ in $\tau_j$; if there is no such element, let the predecessor be the root. So for our first tree in the example, the predecessor of $2$ is $3$, the predecessors of $1$ and $3$ are both $4$, while the predecessor of $4$ is the root. For our second tree, the predecessor of $1$ is the root. For our third tree, the predecessor of $1$ is $2$, the predecessor of $2$ is $3$, the predecessor of $3$ is $4$, and the predecessor of $4$ is the root. Then make a copy of the auxiliary tree and relabel the non-root vertices of the new tree by proceeding as follows, in preorder (i.e. any vertex is processed before its successors): If the label of the current vertex was $k$ in the auxiliary tree, swap its current label with the label that is currently $(1+(\outcome(\pi'_{j}))_k-(\pi'_{j})_k)$th smallest in its subtree. Hence for our first tree in the example, we swap the values of $4$ and $2$ and then swap the values of $3$ and $4$. For our second tree, no action is needed. For our third tree, we swap the values of $4$ and $3$ and then swap the values of $4$ and $1$. This produces a tree whose non-root vertices are labeled from $1$ through $l_j$, and the inversion number of the tree matches the displacement of the corresponding segment of the parking function $\pi_j$. Finally, we replace the labels of the non-root vertices of the forest with the correct labels under the BFS algorithm (recall that they may be read from the inverse of the order permutation associated with the parking function $\pi$), and ensure that relative ordering within each tree is preserved. For our example, $\sigma= 241763598$, we break it into $2417$, $6$, $3598$ and assign these labels to the respective trees. Here again, we use the essential idea in Section \ref{Tutte} that the relative order is what matters, not the vertex label itself. See Figure \ref{illustration4} for a graphical illustration and compare with Figure \ref{illustration1}. This auxiliary tree (hence forest) procedure may be reversed to recover the parking function from the final tree (hence forest); for details see Knuth \cite{Knuth2} and Yan \cite{Yan}. We also note that the critical left-to-right maxima of the parking function may as well be retrieved from the forest, though the process is much less straightforward than for the displacement and will involve simultaneously comparing the auxiliary forest and the final forest.

In addition to explicitly relating the displacement to the inversion number, the above bijection between generalized parking functions and rooted spanning forests has another nice property: Every vertex whose label is the smallest among all the vertices of the subtree rooted at this vertex in the final forest corresponds to a car that successfully parks at its first preference. These are called \textit{lucky} cars and have generated much interest. For example, there are $6$ such cars in Figure \ref{illustration4}, by either reading from the final forest or directly retrieving information from the parking function. Let $L(\pi)$ denote the number of lucky cars. The generating function for $L(\pi)$ was found in Gessel and Seo \cite{GS}, which we adapt to our setting:
\begin{equation}\label{lucky}
\sum_{\pi \in \PF(m, n)} q^{L(\pi)}=(n-m+1)q\prod_{i=1}^{m-1} (i+(n-i+1)q).
\end{equation}
By elementary probability, we can already deduce some interesting statistics for $L(\pi)$. Dividing both sides of (\ref{lucky}) by $(n-m+1)(n+1)^{m-1}$, we recognize the right side as the generating function of $S(m, n)=1+\sum_{i=1}^{m-1} X_i$, with $X_i$ independent,
\begin{equation}
X_i=\left\{
      \begin{array}{ll}
        0, & \hbox{probability $i/(n+1)$,} \\
        1, & \hbox{probability $1-i/(n+1)$.}
      \end{array}
    \right.
\end{equation}
If $m\sim cn$ for some $0<c \leq 1$, then the mean and variance of $S(m, n)$ are
\begin{equation}
\mu=1+\sum_{i=1}^{m-1}\left(1-\frac{i}{n+1}\right) \sim \frac{c(2-c)}{2}n.
\end{equation}
\begin{equation}
\sigma^2=\sum_{i=1}^{m-1} \frac{i}{n+1}\left(1-\frac{i}{n+1}\right) \sim \frac{c^2(3-2c)}{6}n.
\end{equation}
The central limit theorem applies and gives the following.

\begin{theorem}
Take $m$ and $n$ large with $m \sim cn$ for some $0<c \leq 1$. Consider parking function $\pi$ chosen uniformly at random from $\PF(m, n)$. Let $L(\pi)$ be the number of lucky cars. Then for any fixed $x$,
\begin{equation}
\PR\left(\frac{L(\pi)-\frac{c(2-c)}{2}n}{\sqrt{\frac{c^2(3-2c)}{6}n}}\leq x\right) \sim \Phi(x),
\end{equation}
where $\Phi(x)$ is the distribution function of the standard normal.
\end{theorem}

\section{Properties of random parking functions}\label{random}

\subsection{Parking function shuffle}\label{pfs}

Continuing with the last topic of Section \ref{comb}, let us delve deeper into the properties of random parking functions. Through a parking function shuffle construction, we will derive the distribution for single coordinates of a parking function $\pi \in \PF(m, n)$ chosen uniformly at random. Recall that at the beginning of Section \ref{comb}, we explained that parking functions enjoy a nice symmetry: They are invariant under the action of $\Sym_m$ by permuting cars. We will write our results in terms of $\pi_1$ for explicitness, but due to this permutation symmetry, they may be interpreted for any coordinate $\pi_i$. Temporarily fix $\pi_2, \dots, \pi_m$. Let
\begin{equation}
A_{\pi_2, \dots, \pi_m}=\{j: (j, \pi_2, \dots, \pi_m)\in \PF(m, n)\}.
\end{equation}
From the parking scheme, if $k\in A_{\pi_2, \dots, \pi_m}$, then $j\in A_{\pi_2, \dots, \pi_m}$ for all $1\leq j\leq k$, so either $A_{\pi_2, \dots, \pi_m}=[k]$ for some fixed value of $k$ or is empty. This implies that given the last $m-1$ parking preferences, it is sufficient to identify the largest feasible first preference (if exists).

\begin{lemma}\label{nonempty}
If $A_{\pi_2, \dots, \pi_m}=[k]$ is non-empty, then $k \geq n-m+1$.
\end{lemma}

\begin{proof}
We only need to show that if $\pi=(k, \pi_2, \dots, \pi_m)$ is a parking function with $k<n-m+1$, then $\pi'=(k+1, \pi_2, \dots, \pi_m)$ is also a parking function. This follows from (\ref{pigeon}), as $\pi$ and $\pi'$ only differ in the first coordinate and $k+1 \leq n-m+1$, and so
\begin{equation}
\#\{k: \pi_k\leq i\}=\#\{k: \pi'_k\leq i\}, \hspace{.2cm} \forall i=n-m+1, \dots, n.
\end{equation}
\end{proof}

\begin{definition}\label{shuffle}
Let $k \geq n-m+1$. Say that $\pi_2, \dots, \pi_m$ is a \textit{parking function shuffle} of the generalized parking function $\alpha \in \PF(m-n+k-1, k-1)$ and the classical parking function $\beta \in \PF(n-k, n-k)$ if $\pi_2, \dots, \pi_m$ is any permutation of the union of the two words $\alpha$ and $\beta+(k, \dots, k)$. We will denote this by $(\pi_2, \dots, \pi_m) \in \SH(m-n+k-1, k-1, n-k)$.
\end{definition}

\begin{example}
Take $m=8$, $n=10$, and $k=6$. Take $\alpha=(2, 1, 2) \in \PF(3, 5)$ and $\beta=(1, 2, 4, 3) \in \PF(4, 4)$. Then $(2, \underline{7},2, \underline{9}, \underline{10},1, \underline{8}) \in \SH(3, 5, 4)$ is a shuffle of the two words $(2, 1, 2)$ and $(7, 8, 10, 9)$.
\end{example}

\begin{theorem}\label{main1}
Let $k \geq n-m+1$. Then $A_{\pi_2, \dots, \pi_m}=[k]$ if and only if $(\pi_2, \dots, \pi_m) \in \SH(m-n+k-1, k-1, n-k)$.
\end{theorem}

\begin{proof}
``$\Longrightarrow$'' $A_{\pi_2, \dots, \pi_m}=[k]$ is equivalent to saying that $\pi=(k, \pi_2, \dots, \pi_m)$ is a parking function but $\pi'=(k+1, \pi_2, \dots, \pi_m)$ is not. We claim that there can not be any subsequent car with preference $k$ if $\pi_1=k$ is allowed but $\pi'_1=k+1$ is not allowed. Such a car would necessarily park in spots $k+1, \dots, n$ for $\pi$, and consequently it could change places with car $1$ in $\pi'$. So the remaining cars have preference $>k$ or $<k$. Furthermore, the remaining cars with preference $>k$ are exactly $n-k$ in number, and thus those with preference $<k$ are $m-n+k-1$ in number. Let $\alpha$ be the subsequence of $(\pi_2, \dots, \pi_m)$ with value $<k$ and $\beta'$ be the subsequence of $(\pi_2, \dots, \pi_m)$ with value $>k$. Construct $\beta=\beta'-(k, \dots, k)$. It is clear from the above reasoning that $\beta \in \PF(n-k, n-k)$, and also $\alpha \in \PF(m-n+k-1, k-1)$. By Definition \ref{shuffle}, $(\pi_2, \dots, \pi_m) \in \SH(m-n+k-1, k-1, n-k)$.

``$\Longleftarrow$'' We first show that $\pi=(k, \pi_2, \dots, \pi_m)$ is a parking function. From Definition \ref{shuffle}, $\pi$ can be decomposed into three parts: a length $m-n+k-1$ subsequence $\alpha$ with entries $\leq k-1$, one entry $k$, and a length $n-k$ subsequence $\beta'$ with entries $>k$. Moreover, $\alpha$ and $\beta=\beta'-(k, \dots, k)$ are both parking functions. We verify (\ref{pigeon}) case by case. For $n-m+1\leq i<k$,
\begin{equation}
\#\{l: \pi_l \leq i\}=\#\{l: \alpha_l \leq i\}\geq m-n+i,
\end{equation}
using that $\alpha \in \PF(m-n+k-1, k-1)$. This also implies that
\begin{equation}
\#\{l: \pi_l \leq k\}\geq m-n+k.
\end{equation}
Lastly, for $i>k$,
\begin{equation}
\#\{l: \pi_l \leq i\}=\#\{l: \pi_l \leq k\}+\#\{l: 0<\beta_l \leq i-k\}\geq m-n+i,
\end{equation}
using the above and also that $\beta \in \PF(n-k, n-k)$.

Next we show that $\pi'=(k+1, \pi_2, \dots, \pi_m)$ is not a parking function. But this is immediate since the only entries of $\pi'$ that are bounded above by $k$ are those from $\alpha$,
\begin{equation}
\#\{l: \pi'_l \leq k\}=m-n+k-1<m-n+k,
\end{equation}
a contradiction. Combining, we have $A_{\pi_2, \dots, \pi_m}=[k]$.
\end{proof}

\begin{corollary}
The number of parking functions $|\PF(m, n)|$ satisfies a recursive relation:
\begin{equation}
|\PF(m, n)|=\sum_{k=n-m+1}^n k \binom{m-1}{n-k}|\PF(m-n+k-1, k-1)||\PF(n-k, n-k)|.
\end{equation}
\end{corollary}

\begin{proof}
This follows from Theorem \ref{main1}. Let $\pi$ be a parking function. The index $k$ denotes the largest possible $\pi_1$ consistent with $\pi_2, \dots, \pi_m$. By Lemma \ref{nonempty}, $k \geq n-m+1$. The binomial coefficient accounts for the shuffling of the length $m-n+k-1$ subsequence and the length $n-k$ subsequence.
\end{proof}

\begin{corollary}\label{component}
The number of parking functions $\pi \in \PF(m, n)$ with $\pi_1=j$ is
\begin{equation}
(n-m+1) \sum_{s=0}^{\min(n-j, m-1)} \binom{m-1}{s} (n-s)^{m-s-2} (s+1)^{s-1}.
\end{equation}
Note that this quantity stays constant for $j\leq n-m+1$ and decreases as $j$ increases past $n-m+1$ as there are fewer resulting summands.
\end{corollary}

\begin{proof}
If $\pi_1=j$, then $A_{\pi_2, \dots, \pi_m}=[k]$ for some $k\geq \max(j, n-m+1)$. Thus the number of parking functions with $\pi_1=j$ is
\begin{align}\notag
&\sum_{k=\max(j, n-m+1)}^n \binom{m-1}{n-k} |\PF(m-n+k-1, k-1)||\PF(n-k, n-k)|\\
&=\sum_{k=\max(j, n-m+1)}^n \binom{m-1}{n-k} (n-m+1) k^{m-n+k-2} (n-k+1)^{n-k-1} \label{withk} \\
&=(n-m+1) \sum_{s=0}^{\min(n-j, m-1)} \binom{m-1}{s} (n-s)^{m-s-2} (s+1)^{s-1}. \notag
\end{align}
Here the first equality is due to Proposition \ref{number}, and the second equality is a simple change of variables $s=n-k$.
\end{proof}

As a side result of the parking function shuffle construction, we obtain a recurrence relation for the displacement enumerator of generalized parking functions that was introduced in Theorem \ref{inv-disp}. This extends the corresponding result of Kreweras \cite{Kr} for classical parking functions. See Yan \cite{Yan1} for related formulations for the inversion enumerator of rooted forests and complement enumerator of $\vec{x}$-parking functions.

\begin{proposition}\label{side}
The displacement enumerator $D_{m, s}(y)$ of parking functions with $m$ cars and $n$ spots (\ref{displacement}), where $s=n-m+1$, satisfies the recurrence relation
\begin{equation}\label{recurrence}
D_{m, s}(y)=\sum_{i=0}^{m-1} \binom{m-1}{i} (s+y+y^2+\cdots+y^i) D_{i, s}(y)D_{m-1-i, 1}(y),
\end{equation}
with initial condition
\begin{equation}
D_{0, s}(y)=1, \hspace{.2cm} \forall s\geq 1.
\end{equation}
\end{proposition}

\begin{proof}
Consider an arbitrary parking function $\pi \in \PF(m, m+s-1)$. Through the parking function shuffle construction explained in Definition \ref{shuffle} and Theorem \ref{main1}, we demonstrated a bijection from $\pi$ to $i, S, \alpha, \beta, j$, where $0\leq i\leq m-1$, $S$ is an $i$-element subset of $[m-1]$, $\alpha \in \PF(i, i+s-1)$, $\beta \in \PF(m-1-i, m-1-i)$, and $1\leq j\leq i+s$. This implies that
\begin{equation}
\disp(\pi)=\disp(\alpha)+\disp(\beta)+\max(i+1-j, 0),
\end{equation}
which further implies the recurrence relation (\ref{recurrence}).
\end{proof}

\subsection{Single coordinates}
Our asymptotic investigations in this section and the next rely on Abel's extension of the binomial theorem:

\begin{theorem}[Abel's extension of the binomial theorem, derived from Pitman \cite{Pitman} and Riordan \cite{Riordan}]\label{Abel}
Let
\begin{equation}\label{b}
A_n(x, y; p, q)=\sum_{s=0}^n \binom{n}{s} (x+s)^{s+p} (y+n-s)^{n-s+q}.
\end{equation}
Then
\begin{equation}\label{b1}
A_n(x, y; p, q)=A_n(y, x; q, p).
\end{equation}

\begin{equation}\label{b2}
A_n(x, y; p, q)=A_{n-1}(x, y+1; p, q+1)+A_{n-1}(x+1, y; p+1, q).
\end{equation}

\begin{equation}\label{b3}
A_n(x, y; p, q)=\sum_{s=0}^{n} \binom{n}{s}s!(x+s)A_{n-s}(x+s, y; p-1, q).
\end{equation}
Moreover, the following special instances hold via the basic recurrences listed above:
\begin{equation}\label{1}
A_n(x, y; -1, -1)=(x^{-1}+y^{-1})(x+y+n)^{n-1}.
\end{equation}

\begin{equation}\label{2}
A_n(x, y; -1, 0)=x^{-1}(x+y+n)^n.
\end{equation}

\begin{equation}\label{3}
A_n(x, y; -1, 1)=x^{-1} \sum_{s=0}^n \binom{n}{s} (x+y+n)^s (y+n-s) (n-s)!.
\end{equation}
\end{theorem}

%%%
\old{
\begin{lemma}\label{CLT}
Let $X_1, X_2, \dots$ be iid Poisson$(1)$ random variables. Then
\begin{equation}
\PR(X_1+\cdots+X_n \leq n)=\frac{1}{2}+\frac{2}{3}\frac{1}{\sqrt{2\pi n}}+o(\frac{1}{\sqrt{n}}).
\end{equation}
\end{lemma}

\begin{proof}
We apply the continuity-corrected Edgeworth expansion as in Esseen \cite{Esseen} and Kolassa and McCullagh \cite{KM}, so that
\begin{multline}
\PR(X_1+\cdots+X_n \leq n+x\sqrt{n})=\PR(\frac{X_1+\cdots+X_n-n}{\sqrt{n}}\leq x)\\=\Phi(0)+\frac{\exp(-x^2/2)}{\sqrt{2\pi n}}\left(\frac{\mu_3}{6\sigma^3}(1-x^2)+\frac{1}{\sigma}D(x\sigma\sqrt{n})\right)+o(\frac{1}{\sqrt{n}}),
\end{multline}
where $\sigma$ is the standard deviation and $\mu_3$ is the third central moment of $X_1$, $\Phi(x)$ is the distribution function of the standard normal, and $D(x)=\lfloor x \rfloor-x+1/2$
is a periodic function that addresses the discontinuity in the distribution function of the discrete random variable $(X_1+\cdots+X_n-n)/\sqrt{n}$.
\end{proof}

\begin{lemma}\label{LD}
Take $0<c<1$. Let $X_1, X_2, \dots$ be iid Poisson$(1)$ random variables. Then
\begin{equation}
\PR(X_1+\cdots+X_n \leq nc)=\frac{\exp(-n(c\log c-c+1))}{\sqrt{2\pi nc}}\left(1-\frac1{12nc}+O(n^{-2})\right).
\end{equation}
\end{lemma}

\begin{proof}
$X=X_1+\dots+X_n$ is a Poisson$(n)$ random variable. For $c<1$, $\PR(X\le nc)$ is well approximated by
$\PR(X= \lfloor nc\rfloor),$ that is,
$$\PR(X\le nc) = \frac{n^{nc}e^{-n}}{(nc)!}(1+O(n^{-2})).$$
Now apply Stirling's formula.
\end{proof}
}
%%%%

Refocusing on the single coordinate $\pi_1$, we will start by computing its limiting distribution.
It is not surprising that the bulk of the values for $\PR(\pi_1=j)$ are close to uniform $1/n$
but we will also see some intriguing tendencies at the endpoints.
%Using standard arguments, we may further show that the total variation distance converges to zero as $m, n \rightarrow \infty$.
Let $X$ be a random variable satisfying the Borel distribution with parameter $\mu$ ($0\leq \mu\leq 1$), that is, with pdf given by, for $j=1, 2, \dots$,
\begin{equation}
\PR_\mu(X=j)=\frac{e^{-\mu j}(\mu j)^{j-1}}{j!}.
\end{equation}
Denote by $\QR_\mu(j)=\PR_\mu(X\geq j)$. We refer to Stanley \cite{Stanley} for some nice properties of this discrete distribution. The next corollaries show that on the right end, $\pi_1$ approximates a Borel distribution with parameter $m/n$.

\begin{corollary}\label{bo1}
Let $k \geq n-m+1$ and $n-k$ fixed and small relative to $m$ and $n$. For parking function $\pi$ chosen uniformly at random from $\PF(m, n)$, we have
\begin{equation}
\PR(A_{\pi_2, \dots, \pi_m}=[k]) \sim \frac{1}{n}\PR_{m/n}(X=n-k+1),
\end{equation}
where $X$ is Borel-$m/n$ distributed.
\end{corollary}

\begin{proof}
From (\ref{withk}),
\begin{equation}
\PR(A_{\pi_2, \dots, \pi_m}=[k])=\frac{\binom{m-1}{n-k} k^{m-n+k-2} (n-k+1)^{n-k-1}}{(n+1)^{m-1}}.
\end{equation}
Since $n-k$ is small relative to $m$ and $n$, asymptotically, this is
\be
\frac{(n-k+1)^{n-k}}{(n-k+1)!}\left(\frac{k}{n+1}\right)^{m-1} \left(\frac{m}{k}\right)^{n-k} \frac{1}{k}\sim \frac{1}{n(n-k+1)!} \exp\left(-(n-k+1)\frac{m}{n}\right) \left((n-k+1)\frac{m}{n}\right)^{n-k}.\notag
\ee
\end{proof}

%The following corollary is then immediate.

\begin{corollary}\label{boundary1}
Fix $j$ and take $m$ and $n$ large relative to $j$. For parking function $\pi$ chosen uniformly at random from $\PF(m, n)$, we have
\begin{equation}\label{Qj}
\PR(\pi_1=n-j)\sim \frac{1-\QR_{m/n}(j+2)}{n},
\end{equation}
where $\QR_{m/n}(\ell) = \PR_{m/n}(X\ge\ell)$ is the tail distribution function of Borel-$m/n$.
\end{corollary}

\begin{proof}
If $\pi_1=n-j$, then $A_{\pi_2, \dots, \pi_m}=[k]$ for some $k \geq n-j$. From Corollary \ref{bo1}, this implies that
\begin{align}
\PR(\pi_1=n-j)&=\PR(A_{\pi_2, \dots, \pi_m}=[n])+\sum_{k=n-j}^{n-1}\PR(A_{\pi_2, \dots, \pi_m}=[k])\notag \\
&\sim \PR(\pi_1=n)+\frac{1}{n}\left(\QR_{m/n}(2)-\QR_{m/n}(j+2)\right).
\end{align}

Hence we only need to check the boundary case:
\begin{equation}
\PR(\pi_1=n)=\frac{n^{m-2}}{(n+1)^{m-1}}\sim \frac{1}{ne^{m/n}}=\frac{\PR_{m/n}(X=1)}{n}=\frac{1-\QR_{m/n}(2)}{n}.
\end{equation}
\end{proof}

It was shown in Diaconis and Hicks \cite{DH} that for classical parking functions, the asymptotic tendency of $\pi_1$ on the left end mirrors that on the right end. For generalized parking functions, however, it is a completely different story. The next corollaries show that $\pi_1$ stays at a constant value close to $1/n$ for a long time and then exhibits Poisson behavior when it deviates from this value. For explicitness, we will take $m=cn$ for some $0<c<1$, but the strict equality may be relaxed, and similar asymptotics will hold. In particular, the Poisson tendency will stay.

\begin{corollary}\label{bo2}
Let $k \geq n-m+1$ and $k-(n-m+1)$ fixed and small relative to $m$ and $n$ with $m=cn$ for some $0<c<1$. For parking function $\pi$ chosen uniformly at random from $\PF(m, n)$, we have
\begin{equation}
\PR(A_{\pi_2, \dots, \pi_m}=[k]) \sim \exp(n(1-c)/e)\frac{c^{cn-2}}{(1-c)e^c}\frac{1}{n^2}\PR(Y=k-n+m-1),
\end{equation}
where $Y$ is a Poisson$(n(1-c)/e)$ random variable.
\end{corollary}

\begin{proof}
From Corollary \ref{component},
\begin{equation}
\PR(A_{\pi_2, \dots, \pi_m}=[k])=\frac{\binom{m-1}{m-n+k-1} k^{m-n+k-2} (n-k+1)^{n-k-1}}{(n+1)^{m-1}}.
\end{equation}
Since $k-(n-m+1)$ is small relative to $m$ and $n$, asymptotically, this is
\begin{align}
&\frac{1}{(m-n+k-1)!}\left(\frac{(m-1)k}{n-k+1}\right)^{m-n+k-1}\left(\frac{n-k+1}{n+1}\right)^{m-1}\frac{1}{k(n-k+1)} \\
&\sim \frac{1}{(k-n+m-1)!} \left(\frac{n(1-c)}{e}\right)^{k-n+m-1}\frac{c^{cn-2}}{(1-c)e^c}\frac{1}{n^2}. \notag
\end{align}
\end{proof}

%The following corollary is then immediate.

\begin{corollary}\label{boundary2}
Fix $j$ and take $m$ and $n$ large relative to $j$ with $m=cn$ for some $0<c<1$. For parking function $\pi$ chosen uniformly at random from $\PF(m, n)$, we have
\begin{equation}
\PR(\pi_1=n-m+1+j)\sim \exp(n(1-c)/e)\frac{c^{cn-2}}{(1-c)e^c}\frac{1}{n^2}\left(\PR(Y\geq j)-1\right)+\frac{1}{n},
\end{equation}
where $Y$ is a Poisson$(n(1-c)/e)$ random variable.
\end{corollary}

\begin{proof}
If $\pi_1=n-m+1+j$, then $A_{\pi_2, \dots, \pi_m}=[k]$ for some $k \geq n-m+1+j$. From Corollary \ref{bo2}, this implies that
\begin{align}
\PR(\pi_1=n-m+1+j)&=\PR(\pi_1=n-m+1)-\sum_{k=n-m+1}^{n-m+j}\PR(A_{\pi_2, \dots, \pi_m}=[k])\notag \\
&\sim \PR(\pi_1=n-m+1)-\exp(n(1-c)/e)\frac{c^{cn-2}}{(1-c)e^c}\frac{1}{n^2}\PR(Y<j).
\end{align}

Hence we only need to check the boundary case:
\begin{multline}
\PR(\pi_1=1)=\cdots=\PR(\pi_1=n-m+1)=\\ \frac{1}{(n+1)^{m-1}}\sum_{s=0}^{m-1} \binom{m-1}{s} (n-s)^{m-s-2} (s+1)^{s-1}=\frac{n-m+2}{(n-m+1)(n+1)} \sim \frac{1}{n},
\end{multline}
where the second equality uses Abel's extension of the binomial theorem (\ref{1}) with $n \rightarrow m-1$, $x \rightarrow 1$, and $y\rightarrow n-m+1$.
\end{proof}

Next we examine the moments of $\pi_1$. As $c\rightarrow 1$, the correction terms blow up, contributing to the different asymptotic orders between the generic situation $m \lesssim n$ and the special situation $m=n$.
Moments of $\pi_1$ are related to Ramanujan $Q$-functions, as shown in the proof of Theorem \ref{adapted} below; this connection leads us further to Propositions \ref{cov} and \ref{un}. 
\begin{theorem}\label{mean}
Take $m$ and $n$ large with $m=cn$ for some $0<c<1$. For parking function $\pi$ chosen uniformly at random from $\PF(m, n)$, we have
\begin{equation}
\ER(\pi_1^\ell) = \frac{n^\ell}{\ell+1}\left(1+\left(\frac{1-c+\ell(1-3c)}{2(1-c)}\right)\frac1n+O(\frac1{n^2})\right).
\end{equation}
When $m=n$, on the contrary, we have for the first two moments
\begin{equation}
\ER(\pi_1) = \frac{n}{2}\left(1-\sqrt{\frac{\pi}{2n}} + \frac{10}{3n}+O(n^{-1/2})\right),
\end{equation}
and
\begin{equation}
\ER(\pi_1^2) = \frac{n^2}{3}\left(1-\frac{3\sqrt{2\pi}}{4\sqrt{n}} +\frac{11}{2n}+O(n^{-1/2})\right).
\end{equation}
\end{theorem}

\begin{proof}
First consider the case $m=cn,~~c<1$.
From Corollary \ref{component}, for the $\ell$-th moment we have
\begin{align}
&\sum_{j=1}^n j^\ell \#\{\pi \in \PF(m, n): \pi_1=j\}=(n-m+1) \sum_{j=1}^n j^\ell \sum_{s=0}^{\min(n-j, m-1)} \binom{m-1}{s} (n-s)^{m-s-2} (s+1)^{s-1}\\
&= (n-m+1) \sum_{s=0}^{m-1} \binom{m-1}{s} (n-s)^{m-s-2} (s+1)^{s-1} \sum_{j=1}^{n-s} j^\ell \notag \\
&= (n-m+1) \sum_{s=0}^{m-1} \binom{m-1}{s} (n-s)^{m-s-2} (s+1)^{s-1} \frac{(n-s)^{\ell+1}}{\ell+1}\left(1+\frac{\ell+1}{2(n-s)}+O(n^{-2})\right)\label{mom3} \\
&\label{before}=\frac{n-m+1}{\ell+1} \sum_{s=0}^{m-1} \frac{m^s}{s!} n^{m-s+\ell-1}e^{-sm/n} (s+1)^{s-1}\left(1-\frac{s^2+s}{2m}+\frac{s(s-\ell+1)}{n}-\frac{s^2c}{2n}+\frac{\ell+1}{2n}+O(n^{-2})\right).
\end{align}

The tree function $F(z) = \sum_{s=0}^\infty \frac{z^s}{s!}(s+1)^{s-1}$, related to the Lambert function,
satisfies $F(ce^{-c}) = e^c$. By the chain rule its first and second derivatives therefore satisfy
\begin{align}
F'(ce^{-c}) = \frac{e^{2c}}{1-c}, \hspace{1cm} F''(ce^{-c}) = \frac{3-2c}{(1-c)^3}e^{3c}.
\end{align}
Now (\ref{before}) is of the form
\begin{equation}
=\frac{(n-m+1)n^{m+\ell-1}}{\ell+1}\left(\sum_{s=0}^\infty\frac{(ce^{-c})^s}{s!}(s+1)^{s-1}(1+\frac1n(A+Bs+Cs^2)+O(n^{-2}))\right),\label{lamb}
\end{equation}
where $A=\frac{\ell+1}2$, $B=-\frac{1}{2c}-\ell+1$, and
$C=-\frac{1}{2c}+1-\frac{c}{2}.$
Using $F$ this can be written (with $z=ce^{-c}$)
\begin{equation}
=\frac{(n-m+1)n^{m+\ell-1}}{\ell+1}(F(z) + \frac1n\Big(AF(z)+BzF'(z) + C(z^2F''(z)+zF'(z))\Big)+O(n^{-2})).
\end{equation}
Dividing by $|\PF(m,n)| = (n-m+1)(n+1)^{m-1}$ and simplifying we get
$$\frac{n^\ell}{\ell+1}\left(1+\frac1n\left(\frac{1-c+\ell(1-3c)}{2(1-c)}\right)+O(\frac1{n^2})\right)$$
for the $\ell$-th moment.

Next for the $m=n$ case. We recognize that (\ref{mom3}) is asymptotically
$\frac{1}{\ell+1}A_{n-1}(1, 1; -1, \ell)+\frac{1}{2}A_{n-1}(1, 1; -1, l-1)$.
Using (\ref{2}) and (\ref{3}), when $\ell=1$ this is
\begin{align}
&\frac12 (n+1)^{n-1}+\frac12\sum_{s=0}^{n-1}\binom{n-1}{s}(n+1)^s(n-s)(n-s-1)!\notag\\
&=\frac12 (n+1)^{n-1}+\frac{(n-1)!}2\sum_{s=0}^{n-1}\frac{(n+1)^s}{s!}(n-s).
\end{align}
We use the following identity involving the incomplete Gamma function, and its asymptotic expansion (see Tricomi \cite{Tri}):
\begin{equation}
\sum_{j=0}^N\frac{N^j}{j!} = \frac{e^N}{N!}\int_{N}^\infty t^Ne^{-t}dt= e^N\left(\frac12+\frac{2}{3\sqrt{2\pi N}}-\frac1{24N}+O(N^{-3/2})\right).
\end{equation}
With $N=n+1$ in this identity a short calculation yields for the first moment
\begin{equation}
\ER(\pi_1)=\frac{1}{2}+\frac{n}{2}\left(1-\sqrt{\frac{\pi}{2n}} + \frac{7}{3n}+O(n^{-1/2})\right).
\end{equation}
A similar but more involved calculation yields the second moment.
\end{proof}

\subsection{Implications for multiple coordinates}
Hashing with linear probing is an efficient method for storing and retrieving data in computer programming, and can be described as follows. A table with $m$ cells $1, \dots, m$ is set up. We insert $n<m$ items $1, \dots, n$ sequentially into the table, with a hash value $h(i) \in [m]$ assigned to the $i$th item. Two or more items may have the same hash value and hence cause a hash collision, and the linear probing algorithm resolves hash collisions by sequentially searching the hash table for a free location. In detail, starting from an empty table, for $1\leq i\leq n$, we place the $i$th item into cell $h(i)$ if it is empty, and otherwise we try cells $h(i)+1, h(i)+2, \dots$ until an empty cell is found; all positions being interpreted modulo $m$. If the $i$th item is inserted into cell $c(i)$, then its displacement $c(i)-h(i)$ (modulo $m$), which
is the number of unsuccessful probes when this item is placed, is a measure of the cost of inserting it and also a measure of the cost of later finding the item in the table. The total displacement $\sum_{i=1}^n (c(i)-h(i))$ (modulo $m$) is thus a measure of both the cost of constructing the table and of using it. For a comprehensive description of hashing, as well as other storage and retrieval methods, see Knuth \cite[Section 6.4]{Knuth2}. We recognize that the hashing problem is equivalent to our parking problem, and the correspondence between the two settings (hashing vs. parking) is: $n\rightarrow m$ and $m \rightarrow n+1$. In particular, a confined hashing where the $n$ items are successfully inserted into the table leaving the $m$th cell empty exactly corresponds to a parking function with $m$ cars and $n$ available spots.

In Flajolet et al. \cite{Fla}, using moment analysis, limit statistics for the displacement of the hash table were established. They also demonstrated that instead of approaching the Airy distribution as for $m=n$, when $m\sim cn$ for some $0<c<1$, the limit law is Gaussian.
%Their derivation of the Airy distribution follows in spirit the approach of Louchard \cite{Lo} and Tak\'{a}cs \cite{Ta}.
Their findings were further extended by Janson \cite{Janson}. See also Knuth \cite{Knuth1} for related results. As in Theorem \ref{mean} for the asymptotics of moments, a sharp change occurs as $c \rightarrow 1$ for the asymptotics of displacement. The following theorem describes the mean and variance of the displacement of the parking function.

\begin{theorem}[adapted from Flajolet et al. \cite{Fla}]\label{adapted}
Take $m$ and $n$ large with $m=cn$ for some $0<c<1$. For parking function $\pi$ chosen uniformly at random from $\PF(m, n)$, we have
\begin{equation}
\ER(\disp(\pi))\sim \frac{c^2}{2(1-c)}n+\frac{c(c^2-c-1)}{2(1-c)^3}, \hspace{1cm} \Var(\disp(\pi))\sim \frac{c(6c-6c^2+4c^3-c^4)}{12(1-c)^4} n.
\end{equation}
Contrarily, when $m=n$,
\begin{equation}
\ER(\disp(\pi))\sim \frac{\sqrt{2\pi}}{4} n^{3/2}-\frac{7}{6}n+\frac{19\sqrt{2\pi}}{48}n^{1/2}, \hspace{1cm} \Var(\disp(\pi))\sim \frac{10-3\pi}{24} n^3+\frac{184-57\pi}{144}n^2.
\end{equation}
Note that $10/3-\pi$ coincides with the variance of the Airy distribution.
\end{theorem}

\begin{proof}
These asymptotic formulas are adapted from Theorems 1, 2, 4 and 5 of \cite{Fla},
via some identities of the $Q_r$-function, where
\begin{equation}
Q_r(m, n)=\binom{r}{0}+\binom{r+1}{1} \frac{n}{m}+\binom{r+2}{2} \frac{n(n-1)}{m^2}+\cdots
\end{equation}
is a generalized hypergeometric function of the second kind, of which the Ramanujan $Q$-function is a special case with $Q(n)=Q_0(n, n-1)$. For $m=n$, we derive that
\begin{equation}
\ER(\disp(\pi))=\frac{n+1}{2}(Q(n+1)-1)-\frac{n}{2}.
\end{equation}
\begin{equation}
\ER((\disp(\pi))^2)=\frac{n}{12}\left(5n^2+13n+4-(14n+8)\frac{n+1}{n}\left(Q(n+1)-1\right)\right).
\end{equation}
While for $m=cn$ for some $0<c<1$, we have
\begin{equation}
\ER(\disp(\pi))=\frac{cn}{2}\left(Q_0(n+1, cn-1)-1\right).
\end{equation}
\begin{align}
&\ER((\disp(\pi))^2)=\frac{cn}{12}\left[((1-c)n+1)^3+(cn+3)((1-c)n+1)^2\right. \\
&\hspace{1cm}+(8cn+1)((1-c)n+1)+5c^2n^2+4cn-1\notag \\
&\hspace{.5cm}\left.-\left(((1-c)n+1)^3+4((1-c)n+1)^2+(6cn+3)((1-c)n+1)+8cn\right)Q_0(n+1, cn-1)\right]. \notag
\end{align}
The conclusion then follows from standard asymptotic analysis.
\end{proof}

For a parking function $\pi$ chosen uniformly at random from $\PF(n, n)$,
\begin{equation}
\disp(\pi)=\binom{n+1}{2}-(\pi_1+\cdots+\pi_n),
\end{equation}
which implies that, by permutation symmetry,
\begin{equation}
\ER(\disp(\pi))=\binom{n+1}{2}-n\ER(\pi_1).
\end{equation}
This observation is confirmed by the asymptotics established in Theorems \ref{mean} and \ref{adapted}. Furthermore, the following result may be obtained, which confirms the intuition that interactions between different entries of the parking function are relatively weak.  Again due to permutation symmetry, $\pi_1, \pi_2$ may be interpreted as any coordinates $\pi_i, \pi_j$.

\begin{proposition}\label{cov}
Take $n$ large. For parking function $\pi$ chosen uniformly at random from $\PF(n, n)$, we have
\begin{equation}
\Var(\pi_1) \sim \frac{1}{12}n^2+\frac{4-3\pi}{24}n, \hspace{1cm}
\Cov(\pi_1, \pi_2) \sim \frac{8-3\pi}{24}n+\frac{208-57\pi}{144}.
\end{equation}
\end{proposition}

\begin{proof}
By symmetry,
\begin{equation}
\Var(\disp(\pi))=n\Var(X_1)+n(n-1)\Cov(X_1, X_2).
\end{equation}
The asymptotics are then immediate when we combine the results from Theorems \ref{mean} and \ref{adapted}.
\end{proof}

At the moment we do not have an analogous statement for $\PF(m,n)$ (for the variance however
see Theorem \ref{mean}).

For a parking function $\pi$ chosen uniformly at random from $\PF(m, n)$ there are $n-m$ spots that are unattempted by any car. The following proposition gives the average locations of the unattempted spots.

\begin{proposition}\label{un}
Consider parking function $\pi$ chosen uniformly at random from $\PF(m, n)$. Denote by $k_1(\pi), \dots, k_{n-m}(\pi)$ with $0:=k_0<k_1<\cdots<k_{n-m}<k_{n-m+1}:=n+1$ the $n-m$ unattempted spots of $\pi$. We have
\begin{equation}
\ER(k_i(\pi))=i\frac{n+1}{n-m+1}, \hspace{.2cm} \forall i=1, \dots, n-m.
\end{equation}
\end{proposition}

\begin{proof}
Take $k_i=k$, where $k$ ranges from $i$ to $m+i$. Recall the parking function shuffle construction introduced in Section \ref{pfs}. We recognize that the unattempted spot $k$ breaks up the parking function $\pi$ into two components $\alpha$ and $\beta$, with $\alpha \in \PF(k-i, k-1)$ and $\beta \in \PF(m-k+i, n-k)$, and $\pi$ a shuffle of the two. It follows that
\begin{align}
\ER(k_i(\pi))&=\frac{1}{(n-m+1)(n+1)^{m-1}}\sum_{k=i}^{m+i} k \binom{m}{k-i} i k^{k-i-1} (n-m-i+1) (n-k+1)^{m-k+i-1}\\
&=\frac{i(n-m-i+1)}{(n-m+1)(n+1)^{m-1}} \sum_{s=0}^m \binom{m}{s} (i+s)^{s}(n-i-s+1)^{m-s-1} \notag \\
&=i\frac{n+1}{n-m+1}, \notag
\end{align}
where the last step is from applying Abel's extension of the binomial theorem (\ref{b1}) and (\ref{2}) with $n\rightarrow m$, $x\rightarrow n-m-i+1$, and $y\rightarrow i$.
\end{proof}

As in the special case $m=n$, we may show that for a parking function $\pi$ chosen uniformly at random from $\PF(m, n)$,
\begin{align}
\disp(\pi)&=\sum_{i=0}^{n-m} \left[\binom{k_{i+1}-k_i}{2}-\sum_{k_i<\pi_s<k_{i+1}} (\pi_s-k_i)\right] \\
&=\binom{n+1}{2}-(\pi_1+\cdots+\pi_m)-(k_1+\cdots+k_{n-m}), \notag
\end{align}
which implies that, by permutation symmetry,
\begin{equation}
\ER(\disp(\pi))=\binom{n+1}{2}-m\ER(\pi_1)-\frac{1}{2}(n+1)(n-m).
\end{equation}
This observation is again confirmed by the asymptotics established in Theorems \ref{mean} and \ref{adapted}.

Using more esoteric probability, the next theorem then examines the parking phenomenon when the unattempted parking spots are fixed at specified locations.

\begin{theorem}\label{Brownian}
Choose $\pi$ uniformly at random from $\PF(m, n)$ subject to the constraint that fixed $n-m$ spots $k_1, \dots, k_{n-m}$ with $0:=k_0<k_1<\cdots<k_{n-m}<k_{n-m+1}:=n+1$ are unattempted by any car. Define
\begin{equation}
F_i^\pi(x)=\frac{1}{k_{i+1}-k_i} \#\{s: 0<\pi_s-k_i\leq (k_{i+1}-k_i)x\}, \hspace{.2cm} 0\leq i\leq n-m.
\end{equation}
Then for large enough $k_{i+1}-k_i$, we have
\begin{equation}
\sqrt{k_{i+1}-k_i} \left[F_i^{\pi}(x)-x\right]_{0\leq x\leq 1} \weakto (E_x)_{0\leq x\leq 1},
\end{equation}
where $E_x$ is the Brownian excursion and $\weakto$ denotes weak convergence.
\end{theorem}

\begin{proof}
Chassaing and Marckert \cite{CM} established that the queue length in a BFS procedure executed on a uniformly chosen tree of fixed size converges weakly to the Brownian excursion. The conclusion follows when we apply the one-to-one correspondence between the tree components of the rooted spanning forest and the segments of the generalized parking function outlined in Section \ref{pf}.
\end{proof}

Since for fixed unattempted parking spots, the disjoint segments of the parking function are non-interacting, identifying asymptotic statistics related to the entire parking function $\pi$ corresponds to gluing together $n-m+1$ Brownian excursions (since by Proposition \ref{un} typically empty parking spots are far apart), with each piece scaled to an appropriate size. Further, conditioning on the excursion lengths, the different Brownian excursions are independent. For example, suppose that $k_{i+1}-k_i$ is large, then the displacement within segment $[k_i+1, k_{i+1}-1]$ of the parking function $\pi_{i+1}$ satisfies
\begin{equation}
(k_{i+1}-k_i)^{-3/2} \ \disp(\pi_{i+1}) \weakto \int_0^1 E_x \ dx,
\end{equation}
and we recover here partly the convergence of moments of the displacement towards the moments of the Airy law, which was obtained earlier by Flajolet et al. \cite{Fla}. We refer to Diaconis and Hicks \cite{DH} where more research directions (in each Brownian piece) are explored.

%Indeed, using the probability and combinatorial methods detailed in this section, many more asymptotic phenomena of the parking function may be similarly explored.

\subsection{Equality of ensembles} With only minor adaptations, the study in Diaconis and Hicks \cite{DH} of the equivalent features between the micro-canonical and canonical ensembles for classical parking functions will carry through for generalized parking functions. Let $\tilde{\FC}(m, n)=\{f: [m] \rightarrow [n+1]\}$. Thus $|\tilde{\FC}(m, n)|=(n+1)^m$, $\PF(m, n) \subseteq \tilde{\FC}(m, n)$, and
\begin{equation}
|\PF(m, n)|=\frac{n-m+1}{n+1}|\tilde{\FC}(m, n)|.
\end{equation}
As in the proof of Theorem \ref{number}, for every $\pi \in (\mathbb{Z}/(n+1)\mathbb{Z})^m$, there are exactly $n-m+1$ choices for $k \in \mathbb{Z}/(n+1)\mathbb{Z}$ such that $\pi+k(1, \dots, 1)$ (modulo $n+1$) is a parking function with the same pattern as $\pi$ after suitable reordering. As an example, we display one adapted theorem and a result that easily follows from it, which concerns repeats in coordinate values of parking functions.

\begin{theorem}[adapted from Diaconis and Hicks \cite{DH}]\label{al}
Let $\pi \in \PF(m ,n)$ and $f \in \tilde{\FC}(m, n)$ be chosen uniformly at random. Then
\begin{equation}
\PR(Y_1(\pi)=t_1, \dots, Y_{m-1}(\pi)=t_{m-1})=\PR(Y_1(f)=t_1, \dots, Y_{m-1}(f)=t_{m-1})
\end{equation}
for all $m\geq 2$ and $t_1, \dots, t_{m-1} \in \{0, 1\}$, where
\begin{equation}
Y_i(\pi)=\left\{
           \begin{array}{ll}
             1, & \hbox{if $\pi_{i}=\pi_{i+1}$,} \\
             0, & \hbox{otherwise,}
           \end{array}
         \right.
\end{equation}
and $Y_i(f)$ is similarly defined.
\end{theorem}

\begin{theorem}
Take $m$ and $n$ large with $m \sim cn$ for some $0<c<1$. Consider parking function $\pi$ chosen uniformly at random from $\PF(m, n)$. Let $R(\pi)$ be the number of repeats in $\pi$ read from left to right, i.e. $R(\pi)=\#\{i: \pi_i=\pi_{i+1}\}$. Then for any fixed $j$,
\begin{equation}
\PR(R(\pi)=j) \sim \exp(-c)\frac{c^j}{j!}.
\end{equation}
\end{theorem}

\begin{proof}
By definition, $R(\pi)=Y_1(\pi)+\cdots +Y_{m-1}(\pi)$. From Theorem \ref{al}, $Y_1(\pi), \dots, Y_{m-1}(\pi)$ has the same distribution as $Y_1(f), \dots, Y_{m-1}(f)$, where $f$ is chosen uniformly at random from $\tilde{\FC}(m, n)$. Hence the $Y_i(\pi)$'s are iid, with
\begin{equation}
\PR(Y_i(\pi)=1)=\frac{1}{n+1}.
\end{equation}
The Poisson approximation applies and gives the conclusion.
\end{proof}

\section*{Acknowledgements}

Mei Yin acknowledges helpful conversations with Lingjiong Zhu and constructive comments from Mitsuru Wilson.

\end{document}